\documentclass[12pt]{article}
\usepackage{amssymb,amsthm,latexsym,amsmath,amscd}

\numberwithin{equation}{section}
\newtheorem{theorem}{Theorem}
\newtheorem{lemma}{Lemma}
\newtheorem{proposition}{Proposition}
\newtheorem{corollary}{Corollary}
\newtheorem*{theorema}{Main Theorem}

\newtheorem*{conjecture}{Conjecture}

\theoremstyle{definition}

\theoremstyle{definition}
\newtheorem{remark}{Remark}

\theoremstyle{definition}
\newtheorem{definition}{Definition}

\theoremstyle{definition}

\theoremstyle{definition}
\newtheorem{notation}{Notation}

\theoremstyle{definition}

\theoremstyle{definition}

\theoremstyle{definition}

\theoremstyle{definition}

\theoremstyle{definition}

\newcommand{\Real}{\operatorname{Re}}

\newcommand{\Supp}{\operatorname{Supp}}

\newcommand{\dee}{\partial}
\newcommand{\dba}{\bar{\partial}}
\newcommand{\ddba}{\partial \bar{\partial}}

\newcommand{\pa}[2]{\frac{\partial{#1}   } {\partial{#2}            } }

\begin{document}
\title{Sharp Estimates for the $\bar{\partial}$-Neumann Problem \\on Regular Coordinate Domains}
\author{David  W. Catlin\\
Department of Mathematics\\
Purdue University \\ catlin@math.purdue.edu
\and Jae-Seong Cho\\
Department of Mathematics\\
Purdue University \\
cho1@math.purdue.edu}

\date{\today}
\maketitle

\begin{abstract}
This paper  treats subelliptic estimates for the $\bar{\partial}$-Neumann problem on a class of domains known as
regular coordinate domains. Our main result is that the largest subelliptic gain for a regular coordinate domain
is bounded below by a purely algebraic number, the inverse of twice the multiplicity of the ideal associated to a
given boundary point.
\end{abstract}


\section{Introduction}

In this paper, we investigate subelliptic estimates for the $\bar{\partial}$-Neumann problem \cite{FK72} on
 a certain class of smoothly bounded pseudoconvex
domains in $\mathbf{C}^{n+1}.$ Recall that such an estimate holds in a neighborhood $U$ of a given point $z_0$ in $\bar
\Omega$ if
\begin{equation} \label{E:SubEllipEst}
|||\phi|||^2_{\epsilon}\leq
 C\left(||\bar{\partial}\phi||^2+||\bar{\partial}^*\phi||^2+||\phi||^2\right),  \quad \phi \in \mathcal{D}^{0,1}(U).
 \end{equation}
 According to the results of Catlin in \cite{Ca83} and \cite{Ca87}, the above estimate holds if and only if
 the D'Angelo type
 \cite{DA82} of
 $z_0$ is finite, i.e., $T( b\Omega,z_0)<\infty$, where
   $T(b\Omega,z_0)$ measures the maximal order of contact at $z_0$ of any
one-dimensional variety $V$ with the boundary. In  fact, in \cite{Ca83} it is shown that $1/T(b\Omega,z_0)$
is an upper bound for $\epsilon.$ The question of finding sharp lower bounds for $\epsilon$ seems
to be more difficult. In \cite{Ko79} Kohn introduced the method of subelliptic multipliers and showed
in the case of real-analytic boundaries that \eqref{E:SubEllipEst}
 holds for some positive $\epsilon$ when $T(b\Omega,z_0)$ is
 finite. Building on the work of Kohn in \cite{Ko79}, Catlin proved in \cite{Ca87} that \eqref{E:SubEllipEst}
 holds for smooth boundaries
when $\epsilon = T(b\Omega,z_0)^{-n^2A},$ where $A=T(b\Omega,z_0)^{n^2}.$

 In this paper we show that we can obtain much better
$\epsilon$ for a class of domains that are defined by a sum of squares  of holomorphic functions. Specifically,
suppose that $g_1(z),\ldots,g_N(z)$ are holomorphic functions that are defined in a neighborhood of the origin of
the origin in $\mathbb C^n$. We define a domain $\Omega_g \in \mathbb{C}^{n+1}$ by $\Omega_g=\{(z,z_{n+1}): \Real
z_{n+1}+|g_1(z)|^2+\dots +|g_N(z)|^2<0\}.$  Let $\mathcal{O}_n$ denote the ring of germs of holomorphic functions
about the origin in $\mathbb C^n$ and let $I$ denote the ideal in $\mathcal{O}_n $ generated by the germs of
$g_1,\ldots, g_N$. Recall that the multiplicity of $I$ is defined by $m(I)=\dim_{\mathbb C_n} \mathcal O_n/I.$
Considering $T(b\Omega_{g},0) \leq 2m(I)$ in \cite{DA82}, D'Angelo states the following
\begin{conjecture}[D'Angelo \cite{DA93}]
The inequality \eqref{E:SubEllipEst} holds near the origin for the domain $\Omega_g$ with
$\epsilon=\frac{1}{2m(I)}$.
\end{conjecture}

 Siu   \cite{Si07} has shown that for domains of the form $\Omega_g$ one can find a suitable modification of
Kohn's algorithm that also leads to an effective value of $\epsilon$ in terms of the dimension $n$ and the type
$T(b\Omega,0).$

In \cite{DA93} D'Angelo  also introduced the class of regular coordinate domains which are defined as follows:
Let $f_s(z)=f_s(z_1,\dots,z_s)$, $s=1,\dots ,n,$ be holomorphic functions of the first $s$ variables that we can
view as being defined in a neighborhood of the origin in $\mathbf{C}^{n+1}.$ We then define
\begin{equation}\label{E:Regular}
\Omega=\{z\in\mathbf{C}^{n+1} : r(z)=\Real z_{n+1} + \sum_{s=1}^n |f_s(z)|^2 <0  \}.
\end{equation}
The domain $\Omega$ is said to be a \emph{regular coordinate domain} if for each $s=1,\ldots,n,$ there exists a
smallest positive integer $m_s$  such that $\frac{\partial^{m_s} f_s}{\partial z_s^{m_s}}(0)\ne 0.$  It is shown
in  \cite{DA93} that if $I=(f_1,\ldots,f_m),$ then $m(I)=m_1\cdots m_n$  and $T(b\Omega,0) \leq 2 m(I)$. In this
paper we will prove D'Angelo's conjecture for regular coordinate domains.

\begin{theorema}
Let $\Omega\subset \mathbb C^{n+1}$ be a regular coordinate domain defined by functions $f_1,\ldots,f_n$
as above. Then \eqref{E:SubEllipEst} holds near the origin with $\epsilon=\frac{1}{2m_1\cdots m_n}.$
\end{theorema}

We now give an example for which the above value of $\epsilon$ is sharp. Using $f_1(z)=z_1^{m_1}$ and
$f_k(z)=z_k^{m_k}-z_{k-1}$, $k=2,\dots,n$, we  define $\Omega$ as in \eqref{E:Regular}. If we define a curve by
\[
\gamma(\zeta)=(\zeta^{m_2\dots m_n}, \zeta^{m_3\dots m_n}, \dots, \zeta^{m_n}, \zeta, 0).
\]
it is easy to verify that $r(\gamma(\zeta))=|\zeta|^{2m_1\dots m_n}$, and therefore $T(b\Omega,0) \geq 2m_1\cdots
m_n.$ In combination with above-mentioned result in \cite{Ca83}, it follows that $\epsilon \leq
\frac{1}{2m_1\cdots m_n}$. Thus, the value of $\epsilon$ found in the Main Theorem is sharp for this domain.

 The proof of the Main Theorem follows the approach used in \cite{Ca87}  in which it is shown that
in order to prove a subelliptic estimate of order $\epsilon$ near a given boundary point, it suffices to
construct a family of bounded $C^2$ plurisubharmonic functions $\lambda_{\delta}$ with the property that the
Hessian satisfies
\[
H(\lambda_{\delta})(L,\bar{L})\geq c\delta^{-2\epsilon}|L|^2 \quad \mbox{at all points of $U\cap S_{\delta}$},
\]
where $S_{\delta}=\{z\in \bar{\Omega};-\delta<r(z)\leq 0\}$. As in \cite{Ca89} \cite{M92} \cite{NSW}, it suffices
to describe a family of boxes $B(z,\delta)$ for each boundary point $z \in b\Omega \cap U$ and each small $\delta
>0$. Furthermore, one must also construct a bounded plurisubharmonic function $g_{z,\delta}$ compactly supported
in $B(z,\delta)\cap \bar \Omega$ such that its Hessian
 in a slightly smaller box
satisfies lower bounds that correspond to the size of the box. Finally a covering is used to patch the functions
together. In this patching process it is important for the family of boxes to be stable, in the sense that if
$w \in B(z,\delta')\cap b\Omega,$ and if $\delta'$ is comparable in size to $\delta,$ then $B(w,\delta)$ should be
comparable in
size to $B(z,\delta)$..

In the case of regular coordinate domains it is also possible to construct both a family of boxes $B(z,\delta)$
and a family of bounded plurisubharmonic functions $g_{z,\delta}$ supported in $B(z,\delta)\cap\bar \Omega$ with
suitably large Hessian.  The main difficulty comes from the fact that the stability property no longer holds.
Instead we show that there is a set of integer invariants $\mathcal{T}(z,\delta)$ that assume at most a finite
number of values such that if $z$ and $w$ also satisfy, $\mathcal{T}(z,\delta)=\mathcal{T}(w,\delta'),$ then  the
same stability property holds.  Using this additional property a modified version of the covering argument can be
carried out.

\medskip
\noindent \emph{Acknowledgement} The authors would like to express our thanks to John P. D'Angelo. He has
constantly encouraged our work and kindly showed us his recent lecture note about subelliptic multipliers.


\section{Regular Coordinate Domains and their Approximate Systems}\label{S:regularcoordinate}
Let $\Omega$ be a bounded pseudoconvex domain in $\mathbb C^{n+1}$ whose boundary defining function near the
origin is given by 
\begin{equation}\label{E:sumofsquares}
r(z,z_{n+1})=\Real z_{n+1} + \sum_{s=1}^n |f_s(z)|^2,
\end{equation}
where each $f_s$ is a holomorphic function defined near the origin in $\mathbb C^n$ with $f_s(0)=0$. If each
$f_s$ depends only on the first $s$ variables, $z_1,\dots,z_s$, and if for each $s$ there exists a positive
integer $j$ so that 
\begin{equation}\label{E:multiplicity}
\frac{\partial^{j} f_s}{\partial z_s^{j}}(0)\ne 0,
\end{equation}
then, following the terminology in \cite{DA93} and \cite{DA08}, we say that $\Omega$ is a \emph{regular
coordinate domain at the origin} and that $f_1, \dots, f_n$ form a \emph{triangular system at the origin} of
$\Omega$. Let us denote 
\begin{equation}\label{E:amultiplicity}
m_s=\min\left\{j:\frac{\partial^{j} f_s}{\partial z_s^{j}}(0)\ne 0\right\}.
\end{equation}
The power series of $f_s$ at the origin is of the form 
\begin{equation}
f_s(z)= \sum_{j=1}^{\infty} b_{s,j}z_s^j+ \sum_{\alpha \in \mathcal M_s}
c_{s, \alpha} z^{\alpha}, \quad b_{s,m_s}\ne 0
\end{equation}
where  $\mathcal M_s$ denotes the set of multi-indices $\alpha=(\alpha_1,\dots,\alpha_n)$ such that
$\sum_{i=1}^{s-1}\alpha_i \geq 1$ and $\alpha_i=0$ for $i>s$, and $z^{\alpha}=z_1^{\alpha} \dots z_n^{\alpha}$.
\begin{remark}
Let $\mathcal O_n$ denote the ring of germs of holomorphic functions at the origin in $\mathbb C^n$ and let $I$
be the ideal generated by germs of $f_1,\dots,f_n$ in $\mathcal O_n$, described as above.  Set
$m(I)=\dim_{\mathbb C} \mathcal O_n/I$. Then $m(I)=m_1\dots m_n$.
\end{remark}


Since the existence of $\{\lambda_{\delta}\}$ described in \cite{Ca87} is invariant under
a local biholomorphism, without loss of generality we may modify each $f_s$ as follows. By scaling the variables
$z_1,\dots,z_n$, we may assume that the radius of convergence of each $f_s$  at the origin is greater than $2$.
Hence, there exists a neighborhood $U'$ of the origin in $\mathbb C^n$ such that the radius convergence of each
$f_s$ at $p\in U'$ is greater than $\frac{3}{2}$. After multiplying \eqref{E:sumofsquares} by a suitable
constant, we may also assume, by \eqref{E:multiplicity}, that $|b_{s,m_s}| \geq 2, \quad s=1,\dots, n$.

Let $p \in U$ and let us use $u_i=z_i-p_i$ as a coordinate system centered at $p$. We will consider  the power
series of $f_s$ at each $p \in U'$, 
\begin{align}
f^p_s(u) &=f_s(u+p)-f(p) \notag \\
&=\sum_{j=1}^{\infty} b_{s,j}(p)u_s^j + \sum_{\alpha \in \mathcal M_s} c_{s,\alpha}(p) u^{\alpha}, \quad
s=1,\dots,n.\label{E:fppower}
\end{align}
Note that $f^p_1,\dots, f^p_s$, form a triangular system at each $p$ and that
\begin{equation}
b_{s,j}(p)=\frac{1}{j!}\frac{\partial^{j}f_s}{\partial z_s^j}(p), \quad c_{s,\alpha}(p)=\frac{1}{\alpha !
}\frac{\partial^{|\alpha|}f_s}{\partial z^{\alpha}}(p)
\end{equation}
where $\alpha ! =\alpha_1 ! \cdots \alpha_s !$ and $|\alpha|=\alpha_1+\cdots+\alpha_s$.

Since we assumed that the radius of convergence of $f_s^p$, $p\in U'$, is greater than $\frac{3}{2}$, it follows
from Cauchy's estimate that for each $s$
\begin{equation}\label{E:bsjtozero}
\lim_{j\to \infty} |b_{s,j}(p)|=0,  \quad \lim_{|\alpha| \to \infty} |c_{s,\alpha}(p)| =0,  \quad \mbox{uniformly
in $p \in U'$}.
\end{equation}
Also,  $|b_{s,m_s}|\geq 2$ implies that there exists a bounded neighborhood $U \subseteq U'$ of the origin so
that
\begin{equation}\label{E:bjmjsize}
|b_{s,m_s}(p)|\geq 1 \quad  \mbox{for any $p \in U$}.
\end{equation}
It follows from \eqref{E:bsjtozero} that for each $s \geq 1$ there exist constants $B_s\geq 1$ and $C_s \geq 1$,
depending only on $U$, so that
\begin{align}
&|b_{s,j}(p)| \leq B_s, \quad j\geq 1, \quad p \in U \notag\\
&|c_{s,\alpha}(p)| \leq C_s, \quad \alpha \in \mathcal M_s, \quad p \in U.\label{E:coeffi}
\end{align}
We will fix the neighborhood $U$ in this paper. 

\begin{definition}
Let  $p \in U$ and $0<\delta<1$. Define $\tau_1(p,\delta)$ by 
\begin{equation}\label{E:tau1delta}
\tau_1(p,\delta)=\inf_{j\geq 1} \left\{ \left( \frac{\delta^{\frac{1}{2}}}{|b_{1,j}(p)|} \right)^{\frac{1}{j}}
\right\}.
\end{equation}
\end{definition}
\begin{remark}\label{R:welltau1}
Let $B_1$ be the constant in \eqref{E:coeffi}. For $p \in U$ and $0<\delta<1$,
\begin{equation}\label{E:lutau1}
\frac{\delta^{\frac{1}{2}}}{B_1} \leq \tau_1(p,\delta) \leq \delta^{\frac{1}{2m_1}}.
\end{equation}
Indeed, since $0<\frac{\delta^{\frac{1}{2}}}{B_1} <1$, it follows that $\frac{\delta^{\frac{1}{2}}}{B_1} \leq
\left(\frac{\delta^{\frac{1}{2}}}{B_1}\right)^{\frac{1}{j}}$ for any $j\geq 1$. Hence, since $|b_{1,j}(p)|\leq
B_1$, $j\geq 1$, $p \in U$, it follows from \eqref{E:tau1delta} that the first inequality in \eqref{E:lutau1}
holds. The second inequality in \eqref{E:lutau1} results from \eqref{E:bjmjsize}.
\end{remark}

Let $\mu>1$ be a constant. In Section \ref{S:local} we will fix the value of $\mu$, depending only on $n$, and
$m_1,\dots,m_n$. Until then we will consider $\mu>1$ as a parameter. For each $p\in U$, $\mu>1$, $\delta>0$, we
define $\tau_s(p,\mu,\delta)$, $s\geq 1$, inductively. Set
\begin{equation}
\tau_1(p, \mu, \delta)=\tau_1(p,\delta),  \quad p \in U, \quad \mu>1,  \quad \delta>0.
\end{equation}

Let $s\geq 2$ and we assume that for $p\in U$, $\mu>1$, and $\delta>0$,  we have already constructed
$\tau_i(p,\mu,\delta)$, $1\leq i <s$. We construct $\tau_s(p,\mu,\delta)$ as follows: For simplicity, we replace
by $w$  the last variable, $u_s$, in $f^p_s$. Define
\begin{equation}\label{E:DefFspw}
F_s^p(|w|)=\sup_{j\geq 1} |b_{s,j}(p) w^j|,
\end{equation}
and 
\begin{equation}\label{E:DefSdeltap}
C_{s,\delta}^p(\mu,|w|)=\sup_{\alpha \in \mathcal M_s} \left\{ \delta^{\frac{1}{2}},\; \mu \cdot\left|
c_{s,\alpha}(p) \left[\prod_{i=1}^{s-1} (\tau_i(p,\mu,\delta))^{\alpha_i}\right] w^{\alpha_s} \right| \right\}.
\end{equation}
Here we substitute $\tau_i(p,\mu,\delta)$ for $u_i$, $1 \leq i <s$, in each term, $c_{s,\alpha}u^{\alpha}$, in
\eqref{E:fppower}. 
We want to define $\tau_{s}(p,\mu,\delta)$ by 
\begin{equation}\label{E:Deftaus}
\tau_{s}(p,\mu,\delta)= \inf \left\{|w| : F_s^p(|w|) \geq C_{s,\delta}^p(\mu,|w|), \quad 0<|w| <1\right\}.
\end{equation}
We say that $\tau_1(p,\mu,\delta), \dots, \tau_n(p,\mu,\delta)$ form an
\emph{approximate system at $p$ with respect to $\mu>1$ and $\delta>0$}.

In Proposition \ref{P:welldefined} we will show that the set in \eqref{E:Deftaus} is nonempty. To do so we need
the following constants. For each $\mu>1$ set $M_{\mu,1}=1$. Let $C_s$, $2\leq s \leq n$, be the constants in
\eqref{E:coeffi} and define $M_{\mu,s}$, $2\leq s\leq n$ by
\begin{equation}\label{E:Mmus}
M_{\mu,s}= \max \left \{ 1, (\mu C_s M_{\mu,1})^{\frac{1}{m_s}},\dots, (\mu C_s M_{\mu,s-1})^{\frac{1}{m_s}}
\right \},
\end{equation}
Note that each $M_{\mu,s}\geq 1$ depends only on $U$, $\mu$, and $m_1,\dots,m_s$, and that
\begin{equation}\label{E:Msmuinequal}
\mbox{if $\mu_1> \mu_2$, then $M_{\mu_1,s} \geq M_{\mu_2,s}$,\quad $1\leq s \leq n$.}
\end{equation}
Set $\Delta_1=\frac{1}{2}$ and define $\Delta_s$, $2\leq s \leq n$ by
\begin{equation}\label{E:Defdeltas}
\Delta_s=\min \left\{\Delta_1,\dots,\Delta_{s-1}, \frac{1}{2} \left(M_{\mu,s} \right)^{-2m_1 \dots m_{s}}
\right\}.
\end{equation}
Let us denote
\begin{equation}\label{E:deltamu}
\delta_{\mu} = \min\{\Delta_1,\dots,\Delta_n\}.
\end{equation}
Clearly, $\delta_{\mu}$ depend only on $U$, $\mu$, and $m_1,\dots,m_n$, so that
\begin{equation}
0 < \delta_{\mu}=\Delta_n \leq \dots \leq \Delta_1\leq \frac{1}{2},
\end{equation}
and  if
$\mu_1\geq\mu_2$ then $\delta_{\mu_1} \leq \delta_{\mu_2}$.

\begin{proposition}\label{P:welldefined}
For   all
$p \in U$, $\mu>1$, $\delta$ with $0<\delta \leq \delta_{\mu}$, and $s=1,\dots,n$,
the set in \eqref{E:Deftaus} is nonempty.
Furthermore, each $\tau_s(p,\mu,\delta)$ satisfies
\begin{equation}\label{E:lutaus}
\tau_s(p,\mu,\delta) \leq M_{\mu,s} \delta^{\frac{1}{2m_1\dots m_s}}<1.
\end{equation}
\end{proposition}
\begin{proof}
Fix $\mu>1$ and let $\Delta_s$ be the constants in \eqref{E:Defdeltas}. We will show the proposition by induction
on $s$. It follows from Remark \ref{R:welltau1} that \eqref{E:lutaus} holds for $s=1$. Assume inductively that
$\tau_i(p,\mu,\delta)$ is well-defined for $p\in U$, $0 <\delta \leq \Delta_i$, and $i<s$, and that
\eqref{E:lutaus} holds when $i<s$.

Since  $|c_{s,\alpha}(p)| \leq C_s$ for $\alpha \in \mathcal M_s$, it follows from \eqref{E:DefSdeltap} that if
$p \in U$ and $|w| \leq 1$, then 
\begin{equation}\label{E:deltacssup}
\delta^{\frac{1}{2}} \leq C_{s,\delta}^{p}(\mu, |w|) \leq \sup_{\alpha \in \mathcal M_s} \left\{
\delta^{\frac{1}{2}},\; \mu \cdot C_s \cdot \prod_{i=1}^{s-1} (\tau_i(p,\mu,\delta))^{\alpha_i}  \right\}
\end{equation}
If $0<\delta\leq \min\{\Delta_1,\dots,\Delta_{s-1}\}$, then  for any $p \in U$ and $|w|\leq 1$,
\begin{align}
C_{s,\delta}^{p}(\mu,|w|) &\leq \sup \left\{ \delta^{\frac{1}{2}}, \; \mu C_s\tau_1(p,\mu,\delta), \dots, \mu C_s
\tau_{s-1}(p,\mu,\delta) \right\} \notag\\
& \leq \sup \left\{ \delta^{\frac{1}{2}}, \;\mu  C_s  M_{\mu,1} \delta^{\frac{1}{2m_1}}, \dots, \;\mu C_s
M_{\mu,s-1}\delta^{\frac{1}{2m_1\dots m_{s-1}}} \right\} \notag\\
& \leq (M_{\mu,s})^{m_s} \delta^{\frac{1}{2m_1\dots m_{s-1}}} . \label{E:CpsdeltaMs}
\end{align}
Indeed, the first inequality results from \eqref{E:deltacssup}, since $\alpha_1+\dots+\alpha_{s-1}\geq 1$ for
$\alpha \in \mathcal M_s$, and $\tau_i(p,\mu,\delta)<1$, $1\leq i<s$ by \eqref{E:lutaus}. The second inequality
is obtained from \eqref{E:lutaus}, and the last inequality follows from \eqref{E:Mmus}.

We combine \eqref{E:deltacssup} and  \eqref{E:CpsdeltaMs} to obtain that if  $0<\delta \leq \Delta_s$, then
\begin{equation}\label{E:Cvalue}
\delta^{\frac{1}{2}}\leq C_{s,\delta}^p(\mu, |w|)\leq (M_{\mu,s})^{m_s} \delta^{\frac{1}{2m_1\dots m_{s-1}}}
 < 1, \quad p\in U, \quad |w|\leq 1.
\end{equation}
In fact, the last inequality follows from \eqref{E:Defdeltas}, that is,
\[
\left(M_{\mu,s}\right)^{m_s} \Delta_s^{\frac{1}{2m_1\dots m_{s-1}}} \leq
\left(\frac{1}{2}\right)^{\frac{1}{2m_1\dots m_{s-1}}} <1.
\]
Note that since each term of $F_s^p(|w|)$,  is a continuous monotone increasing function in $|w|$, $F_s^p(|w|)$
in \eqref{E:DefFspw} is a continuous monotone increasing function in $|w|$. Furthermore, $F_s^p(|w|)$ satisfies
\begin{equation}\label{E:Fat0}
F_s^p(0)=0, \quad F_s^p(1)\geq 1.
\end{equation}
Indeed, since
\[
F_s^p(|w|)\geq |b_{s,m_s}(p) w^j|=|b_{s,m_s}(p)|\geq 1, \quad \mbox{if} \quad |w|=1,
\]
the inequality in \eqref{E:Fat0} holds.

Combining \eqref{E:Cvalue}, \eqref{E:Fat0}, and the intermediate value theorem, we conclude  that for each $p \in
U$ and $\delta$ with $0<\delta<\Delta_s$, there exists $w_0$ with $0<|w_0|<1$ so that 
\begin{equation}\label{E:w0equ}
F_s^p(|w_0|)= (M_{\mu,s})^{m_s}\delta^{\frac{1}{2m_1\dots m_{s-1}}} \geq C_{s,\delta}^p(\mu,|w_0|) \geq
\delta^{\frac{1}{2}}.
\end{equation}
Hence, the set in \eqref{E:Deftaus} is nonempty and  contains $|w_0|$. Furthermore, it follows from
\eqref{E:Deftaus} that  $\tau_s(p,\mu,\delta)$ satisfies 
\begin{equation}\label{E:tausw0}
\tau_s(p,\mu, \delta) \leq |w_0|.
\end{equation}
Since $|b_{s,m_s}(p)| \geq 1$, it follows from \eqref{E:DefFspw}, \eqref{E:Cvalue}, and \eqref{E:w0equ} that for
each $p \in U$ and $\delta$ with $0<\delta<\Delta_s$
\[
|w_0^{m_s}| \leq |b_{s,m_s}(p)w_0^m| \leq F_s^p(|w_0|)=( M_{\mu,s})^{m_s}\delta^{\frac{1}{2m_1\dots m_{s-1}}}<1.
\]
Hence, it implies that
\begin{equation}\label{E:w0final}
|w_0| \leq M_{\mu,s} \delta^{\frac{1}{m_1\dots m_s}}<1.
\end{equation}
Combine \eqref{E:tausw0} and \eqref{E:w0final} to obtain that for $p \in U$  and $\delta$ with
$0<\delta<\Delta_s$ 
\begin{equation}\label{E:Uniupper}
\tau_s(p,\mu, \delta) \leq |w_0| \leq M_{\mu,s} \delta^{\frac{1}{m_1\dots m_s}}<1.
\end{equation}
\end{proof}
\begin{proposition}\label{P:Lowertaus}
Let $p \in U$, $\mu>1$, and $0<\delta\leq \delta_{\mu}$, and let $B_s$ be the constants in \eqref{E:coeffi}. For
each $s=1,\dots,n$, $\tau_s(p,\mu,\delta)$ satisfies
\begin{equation}\label{E:Lowertaus}
\frac{\delta^{\frac{1}{2}}}{B_s} \leq \tau_s(p,\mu,\delta).
\end{equation}
\end{proposition}
\begin{proof}
It follows from \eqref{E:DefSdeltap} and \eqref{E:Deftaus} that
\begin{equation}\label{E:Fsbigdelta}
F^p_s(\tau_s(p,\mu,\delta)) = C_{s,\delta}^p(\mu, \tau_s(p,\mu,\delta)) \geq \delta^{\frac{1}{2}}, \quad p \in U,
\quad 0<\delta\leq \delta_{\mu}.
\end{equation}
Let $|w_0|=\tau_s(p,\mu,\delta)$. Suppose that there exist $p \in U$, $\mu>1$, and $\delta$ with $0<\delta\leq
\delta_{\mu}$, for which \eqref{E:Lowertaus} does not hold.  We will obtain a contradiction to
\eqref{E:Fsbigdelta} by showing that
\begin{equation}\label{E:Fssmalldelta}
F^p_s(|w_0|) < \delta^{\frac{1}{2}}, \quad p \in U, \quad 0<\delta\leq \delta_{\mu}.
\end{equation}

Since we assumed that $|w_0|<\frac{\delta^{\frac{1}{2}}}{B_s}$, there is a constant $c$ with $0<c<1$ such that
\begin{equation}\label{E:w0c}
|w_0| =c \frac{\delta^{\frac{1}{2}}}{B_s}<1,
\end{equation}
where the second inequality results from $B_s\geq 1$ and $\delta_{\mu}<1$. Since
$|w_0|^j \leq |w_0|=c \frac{\delta^{\frac{1}{2}}}{B_s},$
and since $|b_{s,j}(p)|\leq B_s$ for all $j\geq 1$,  it follows  that
\begin{equation}
|b_{s,j}(p)w_0^j|\leq |b_{s,j}(p)||w_0| \leq B_s c\frac{\delta^{\frac{1}{2}}}{B_s} =c \delta^{\frac{1}{2}}, \quad
j\geq 1.
\end{equation}
Therefore, it follows from \eqref{E:DefFspw} that
$F^p_s(|w_0|))\leq c \delta^{\frac{1}{2}}< \delta^{\frac{1}{2}},$
which contradicts to \eqref{E:Fsbigdelta}.
\end{proof}
\begin{remark}\label{R:UniSmall}
Since $\tau_s(p,\mu,\delta) \leq M_{\mu,s} \delta^{\frac{1}{2m_1\dots m_s}}$, $1 \leq s \leq n$, and since
$M_{\mu,s}$ depends only on $U$, $\mu$, and $m_1,\dots,m_s$, it follows that for each fixed $\mu$,
$\lim_{\delta \to 0} \tau_s(p,\mu,\delta)=0$ uniformly in $p \in U$.
\end{remark}

\section{Invariants of a triangular system}\label{S:types}
In this section we introduce two kinds of integer invariants, called \emph{dominant} and \emph{mixed types}, at
$p \in U$ with respect to  $\mu>1$ and $\delta$ with $0 < \delta<\delta_{\mu}$. After shrinking $\delta_{\mu}$
for each $\mu>1$, we will show that there are only finitely many both dominant and mixed types.
\begin{definition}[dominant types]\label{D:DominantType}
Let $p \in U$, $\mu>1$, and $0<\delta\leq \delta_{\mu}$.  For simplicity we write $\tau_s$ for
$\tau_s(p,\mu,\delta)$, $1\leq s \leq n$. Define 
\begin{equation}
J_s(p,\mu,\delta)= \min \{ j\geq 1 :F_s^p(\tau_s)=|b_{s,j}(p)\tau_s^j|\}, \quad 1\leq s \leq n.
\end{equation}
We say that  $J_s(p,\mu,\delta)$ is the \emph{$s$-th dominant type at $p\in U$ with respect to $\mu$ and
$\delta$}.
\end{definition}
\begin{definition}[mixed types] \label{D:mixedtype}
It follows from \eqref{E:Deftaus} that for each $s=2,\dots, n$ we have the following two cases:
\begin{itemize}
\item [(i)]  There exists a multi-index $\alpha \in \mathcal M_s$ such that $\tau_s(p,\mu,\delta)$ satisfies
\begin{equation}\label{E:mixed}
F_s^p(\tau_s)= |\mu \cdot c_{s,\alpha}(p) \tau_1^{\alpha_1} \dots \tau_s^{\alpha_s}| \geq \delta^{\frac{1}{2}}.
\end{equation}
\item [(ii)] For all $\alpha \in \mathcal M_s$
\begin{equation}\label{E:nomixed}
F_s^p(\tau_s)=\delta^{\frac{1}{2}} > |\mu \cdot c_{s,\alpha}(p) \tau_1^{\alpha_1} \dots \tau_s^{\alpha_s}|.
\end{equation}
\end{itemize}
For the first case we write $K_{s}(p,\mu,\delta)$ for a multi-index $K=(k_1^s,\dots,k_n^s) \in \mathcal M_s$
satisfying
\begin{itemize}
\item [(i)] \eqref{E:mixed} holds for $K$, and \item [(ii)] if   \eqref{E:mixed} holds for
$\alpha=(\alpha_1,\dots,\alpha_s,0,\dots,0)$, then the $s$-th indices satisfy $k_s^s \leq \alpha_s$,
\end{itemize}
In this case, we say that $K_s(p,\mu,\delta)$ is  a \emph{$s$-th mixed type at $p$ with respect to $\mu$ and
$\delta$}. 
In the second case, we say that there is \emph{no $s$-th mixed type at $p$ with respect to $\mu$ and $\delta$}
and we simply write $K_s(p,\mu,\delta)=(0,\dots,0)$.
\end{definition}

\begin{proposition}\label{P:FinitePure}
For each $\mu>1$ there exists $\delta_{\mu}'$, depending only on $\mu$ and $m_1,\dots,m_n$, so that if $p \in U$
and $0<\delta\leq \delta'_{\mu}$, then
\begin{equation}\label{E:FiniteDomi}
1\leq J_s(p,\mu,\delta) \leq m_s, \quad 1\leq s \leq n.
\end{equation}
Furthermore, $\delta'_{\mu}$ satisfies $0< \delta_{\mu}' \leq \delta_{\mu}$, where $\delta_{\mu}$ is constructed
in \eqref{E:deltamu}.
\end{proposition}
In order to prove Proposition \ref{P:FinitePure}, we at first state an elementary fact about a set of monomials
with positive coefficients in a real positive variable. We shall divide the positive real line into a finite set
of intervals on which one of monomials dominates the others.
\begin{lemma}\label{L:Mono}
Let $g_j(x)$ be monomials in a real variable $x>0$ such that 
\begin{equation}
g_j(x)=C_j x^{m_j}, \quad j=1,\dots,N,
\end{equation}
where $0<m_1< \dots< m_N$ and $C_j>0$. Let
\begin{equation}
g(x)=\max\{g_j(x) : 1 \leq j \leq N\}.
\end{equation}
Then there exist  integers $j_k$, $k=1,\dots,q$  such that 
\begin{equation}
0 < j_1 <\dots <j_q=N,
\end{equation}
and points $\bar{x}_{k}$, $k=0, \dots,q$ with $0=\bar{x}_{0}< \bar{x}_{1} < \dots < \bar{x}_{q}=\infty$ and 
\begin{equation}
g(x)=g_{j_k}(x), \quad \bar{x}_{k-1}\leq x \leq \bar{x}_{k}.
\end{equation}
Furthermore, $j_k$ satisfies
\begin{equation}\label{E:jkminimum}
j_k=\min\{j: g(\bar x_k)=g_j(\bar x_k)\},
\end{equation}
and  we have $g_{j_k}(\bar{x}_{k})=g_{j_{k+1}}(\bar{x}_{k})$, $k=1,\dots,q-1$.
\end{lemma}
\begin{lemma}\label{L:Dominance}
Let $h(w)=\sum_{j=1}^{\infty} b_j w^j$ be  holomorphic  in $\{w: |w| <\frac{3}{2}\}$. Let
\begin{equation}
F(|w|)=\sup\{|b_jw^j|:j\geq 1\}.
\end{equation}
Suppose
that $|b_m|\geq 1$ and that there exists a constant $B\geq 1$ such that $|b_j| \leq B$, $j\geq 1$.
If $|w|\leq B^{-1}$, then $F(|w|)$ ie determined by the first $m$-terms, that is, 
\begin{equation}\label{E:finiteFw}
F(|w|)=\max\{|b_jw^j| : 1 \leq j \leq m \}, \quad |w|\leq B^{-1}.
\end{equation}
\end{lemma}
\begin{proof}
If  $|w| \leq  B^{-1}$, then for any $j>m$
\[
|w|^{j-m} \leq  |w|\leq B^{-1}.
\]
Since $|b_m|\geq 1$, it follows that if $j>m$, then for $|w|\leq B^{-1}$
\begin{equation*}
|b_j||w^j|  \leq |b_j| |w|^{j-m} |w|^m \leq B B^{-1} |w|^m \leq |w|^m \leq |b_m||w|^m.
\end{equation*}
Hence, we obtain \eqref{E:finiteFw}.
\end{proof}
\begin{proof}[Proof of Proposition \ref{P:FinitePure}]
Let $p \in U$ and $s=1,\dots, n$. Recall that
\[
F_s^p(|w|)=\sup_{j\geq 1} \{ |b_{s,j} (p) w^j|\},
\]
and  that $|b_{s,m_s}(p)| \geq 1$ and $|b_{s,j}(p)| \leq B_s$ with $B_s \geq 1$ in \eqref{E:coeffi}. Hence, by
Lemma \ref{L:Dominance} we obtain that for each $s=1,\dots,n$,
\begin{equation}\label{E:Fspfinite}
F^p_s(|w|)=\max_{1\leq j \leq m_s}|b_{s,j}(p)w^j|, \quad |w| \leq B_s^{-1}.
\end{equation}
Recall the constant $M_{\mu,s}$ in \eqref{E:Mmus}. For each $\mu>1$ we define $\delta'_{\mu}$ by
\begin{equation}\label{E:deltaprimemu}
\delta'_{\mu}=\min\{ \delta_{\mu}, (B_1 M_{\mu,1})^{-2m_1}, \dots, (B_n M_{\mu,n})^{-2m_1\cdots m_n}\},
\end{equation}
Clearly, $\delta'_{\mu}$ depends only on $\mu$, and $m_1,\dots,m_n$, and satisfies $0< \delta'_{\mu} \leq
\delta_{\mu}$. It follows from Proposition \ref{P:welldefined} that $\tau_s(p,\mu,\delta)$ is well-defined for
$p\in U$, $\mu>1$, and $\delta$ with $0<\delta \leq \delta'_{\mu}$. Furthermore,  we obtain that for each
$s=1,\dots,n$
\begin{equation}\label{E:tausBs}
\tau_s(p,\mu,\delta)  \leq M_{\mu,s} \delta^{\frac{1}{2m_1\cdots m_s}} \leq M_{\mu,s} (B_s
M_{\mu,s})^{-\frac{2m_1\cdots m_s}{2m_1\cdots m_s}}=B_s^{-1}.
\end{equation}
In fact, the first inequality results from \eqref{E:lutaus}, and the second inequality follows from $0 < \delta
\leq \delta'_{\mu} \leq (B_sM_{\mu,s})^{-2m_1\dots m_s}$ by \eqref{E:deltaprimemu}.

Combining \eqref{E:Fspfinite} and \eqref{E:tausBs}, we obtain that for each $s=1,\dots,n$, there exists $j$ with
$1\leq j \leq m_s$ so that $F_s^p(\tau_s)=|b_{s,j} \tau_s^j|$. Hence, by Definition \ref{D:DominantType} we
obtain \eqref{E:FiniteDomi}.
\end{proof}
To prove the finiteness of mixed types we need the following lemma.
\begin{lemma}\label{L:tauratio}
Let $p \in U$, $\mu > 1$, and let $\delta'_{\mu}$ be in \eqref{E:deltaprimemu}. If $0<a<1$ and $0<\delta\leq
\delta'_{\mu}$, then 
\begin{equation}\label{E:assumptiontaui}
\tau_s(p,\mu, a\delta) \leq a^{\frac{1}{2m_1\dots m_s}} \tau_s(p,\mu, \delta).
\end{equation}
\end{lemma}
\begin{proof}
We will prove the lemma by induction on $s$. Fix $p \in U$ and $\mu >1$, and replace $\tau_{s}(p,\mu,\delta)$ by
$\tau_s(\delta)$ in this proof. It follows from Proposition \ref{P:FinitePure} that
\begin{equation*}
\tau_1(a\delta) =\min_{1 \leq j \leq m_1} \left(
\frac{(a\delta)^{\frac{1}{2}}}{|b_{1,j}(p)|}\right)^{\frac{1}{j}} \leq a^{\frac{1}{2m_1}}  \min_{1 \leq j \leq
m_1} \left( \frac{\delta^{\frac{1}{2}}}{|b_{1,j}(p)|}\right)^{\frac{1}{j}}.
\end{equation*}
Hence, \eqref{E:assumptiontaui} holds for $s=1$.

Let $s$ be an integer with $2 \leq s \leq n$. Assume inductively that when $1\leq i \leq s-1$,
\eqref{E:assumptiontaui} holds for $0<\delta\leq \delta'_{\mu}$ and $0<a<1$. 
Note that $\alpha_1+\dots+\alpha_{s-1}\geq 1$ for $\alpha \in \mathcal M_s$ and that
\begin{equation}
C^p_{s,a\delta} (\mu, |w|) = \sup_{\alpha \in \mathcal M_s} \left \{  (a\delta)^{\frac{1}{2}},\; \mu
\left|c_{s,\alpha}(p) \left(\prod_{i=1}^{s-1}(\tau_i(a\delta))^{\alpha_i}\right) w^{\alpha_s} \right| \right\}.
\end{equation}
Hence, since we assumed that  \eqref{E:assumptiontaui} holds for $1\leq i \leq s-1$,  it follows that 
\begin{equation}\label{E:cadeltacdelta}
C_{s,a\delta}^p(\mu, |w|) \leq a^{\frac{1}{2m_1\dots m_{s-1}}} C_{s,\delta}^p(\mu,|w|).
\end{equation}
Set $|w_0|=\tau_s(\delta)$. Since $C_{s,a\delta}^p(\mu,|w|)$ is increasing in $|w|$, it follows from
\eqref{E:cadeltacdelta} that  
\begin{equation}\label{E:Caadelta}
C_{s,a\delta}^p(\mu,|w|) \leq a^{\frac{1}{2m_1 \dots m_{s-1}}} C_{s,\delta}(\mu,|w_0|), \quad |w|\leq |w_0|.
\end{equation}

Let $J_s$ denote the $s$-th dominant type at $p$ with respect to $\mu$ and  $\delta$. Since by \eqref{E:Deftaus}
\begin{equation}\label{E:Cdeltab}
C_{s,\delta}^p(\mu,|w_0|) = |b_{s,J_s} w_0^{J_s}|,
\end{equation}
and since $J_s \leq m_s$, it follows from \eqref{E:Caadelta} and \eqref{E:Cdeltab} that if $|w| \leq |w_0|$, then
\begin{align}
C_{s,a\delta}^p (\mu,|w|) &\leq a^{\frac{1}{2m_1\dots m_{s-1}}} |b_{s,J_s}(p) w_0^{J_s}| \notag\\
& \leq a^{\frac{1}{2m_1\dots m_{s-1}}} \left| \frac{w_0}{w} \right|^{J_s} |b_{s,J_s}(p) w^{J_s}| \notag\\
& \leq a^{\frac{1}{2m_1\dots m_{s-1}}} \left|\frac{w_0}{w} \right|^{m_s} F_s^p(|w|)\label{E:CadeltaF}
\end{align}
Set 
\begin{equation*}
|w_1| = a^{\frac{1}{2m_1 \dots m_s}} |w_0| <|w_0|,
\end{equation*}
that is, $a^{\frac{1}{2m_1 \dots m_{s-1}}}\left| \frac{w_0}{w_1} \right|^{m_s}=1$. Hence, it follows from
\eqref{E:CadeltaF} that 
\begin{equation*}
C_{s,a\delta}^p(\mu,|w_1|) \leq F_s^p(|w_1|).
\end{equation*}
Therefore, by definition of $\tau_s(a\delta)$ we have 
\begin{equation*}
\tau_s(a\delta) \leq |w_1| =a^{\frac{1}{2m_1\dots m_s}} \tau_s(\delta).
\end{equation*}
\end{proof}

\begin{proposition}\label{P:FiniteMixed}
For each $\mu>1$ there exists $\tilde \delta_{\mu}$ with $\tilde\delta_{\mu} \leq \delta'_{\mu}$ so that the
following property holds: Let $2\leq s\leq n$. Suppose that there exists a $s$-th mixed type, denoted by
 $K_s(p,\mu,\delta)=(k_1^s, \dots, k_s^s,0,\dots,0)$, at $p\in U$ with respect to $\mu$ and $\delta$
 with $0<\delta\leq \tilde \delta_{\mu}$. Then, $K_s(p,\mu,\delta)$ must
satisfy
\begin{equation}\label{E:MixedFinite}
\frac{k_1^s}{m_1} + \frac{k_2^s}{m_1m_2} + \dots + \frac{k_s^s}{m_1\dots m_s} \leq 1.
\end{equation}
Furthermore, the $s$-th index, $k_s^s$ satisfies
\begin{equation}\label{E:kslessJs}
k_s^s < J_s,
\end{equation}
where $J_s$ is the dominant type at $p$ with respect to $\mu$ and $\delta$.
\end{proposition}

\begin{proof} [Proof of \eqref{E:MixedFinite}]
Fix any $\mu$ with $\mu>1$. Suppose that $\alpha \in \mathcal M_s$ satisfies 
\begin{equation}\label{E:power12}
\frac{\alpha_1}{m_1}+\frac{\alpha_2}{m_1m_2} + \dots + \frac{\alpha_s}{m_1 \dots m_s}
>1.
\end{equation}
We want to find $\tilde\delta_{\mu}$ so that for any $\delta$ with $0<\delta\leq \tilde\delta_{\mu}$,
\begin{equation}\label{E:toshow}
 \left|\mu \cdot c_{s,\alpha}(p) \prod_{1\leq i \leq s}[\tau_i(p,\mu,\delta)]^{\alpha_i}
\right| <
 \delta^{\frac{1}{2}}.
\end{equation}
Combining Definition \ref{D:mixedtype} and \eqref{E:toshow}, we conclude that
any $\alpha \in \mathcal M_s$ with \eqref{E:power12} cannot be a mixed type at $p$ with
respect to $\mu$ and $\delta$.

Let $\delta=a\delta'_{\mu}$ with $0<a <1$,
where $\delta'_{\mu}$ is in \eqref{E:deltaprimemu}. Then we obtain that
\begin{align}
 \left|\mu \cdot c_{s,\alpha}(p) \prod_{1\leq i \leq s}[\tau_i(p,\mu,\delta)]^{\alpha_i} \right|
 & \leq \mu C_s \prod_{1\leq i\leq s} [\tau_i(p,\mu,a\delta'_{\mu})]^{\alpha_i} \notag\\
 & \leq \mu C_s \prod_{1\leq i \leq s} \left\{ a^{\frac{\alpha_i}{2m_1\cdots m_i}}
 [\tau_i(p,\mu,\delta)]^{\alpha_i} \right\} \notag\\
 & \leq \mu C_s a^{\frac{\alpha_1}{2m_1}+ \dots + \frac{\alpha_s}{2m_1\cdots m_s}}\label{E:mixedmucs}
\end{align}
In fact, the first inequality is obtained by  $|c_{s,\alpha}(p)| \leq C_s$. Lemma \ref{L:tauratio} implies the
second inequality. The third inequality follows from $\tau_i(p,\mu,\delta'_{\mu})<1$ in Proposition
\ref{P:welldefined}.

We now show that if $a$ is
sufficiently small, then
the last term in \eqref{E:mixedmucs} is less than $\delta^{\frac{1}{2}}$.
For each $s=1,\dots,n$,  let us denote
\begin{equation}\label{E:Asindex}
A_s=\inf\left\{\sum_{i=1}^s \frac{\alpha_i}{m_1\cdots m_i} \mid \alpha \in \mathcal M_s, \; \sum_{i=1}^s
\frac{\alpha_i}{m_1\cdots m_i} >1 \right\}.
\end{equation}
Since the set $\{\alpha \in \mathcal M_s : 1 < \sum_{i=1}^s \frac{\alpha_i}{m_1\cdots m_i}<2 \}$ is finite, it
follows that
\begin{equation}\label{E:Asbigger1}
A_s>1, \quad s=1,\dots,n.
\end{equation}
Let
\begin{equation}\label{E:deltadoublemus}
\Delta'_{\mu,s}= \left( \frac{\delta'_{\mu}}{(2\mu C_s)^2} \right)^{\frac{1}{A_s-1}} \delta'_{\mu}, \quad 1\leq s
\leq n,
\end{equation}
where $\delta'_{\mu}$ is constructed in \eqref{E:deltaprimemu} and $C_s\geq 1$  in \eqref{E:coeffi}. We now
define $\tilde \delta_{\mu}$ by
\begin{equation}\label{E:tildeltamu}
\tilde \delta_{\mu}= \min \{ \Delta'_{\mu,s} : 1\leq s \leq n\}.
\end{equation}
Clearly, $\tilde \delta_{\mu}$ depends only on $\mu$, and $m_1,\dots,m_n$, and  satisfies $0<\tilde \delta_{\mu}
\leq \delta'_{\mu}$.
If $0<\delta=a\delta'_{\mu}\leq \tilde\delta_{\mu}$, then since $0<a<1$ and
$\frac{\alpha_1}{m_1} + \dots + \frac{\alpha_s}{m_1\cdots m_s}\geq A_s
>1$,
the last term in \eqref{E:mixedmucs} satisfies
\begin{equation}
\mu C_s a^{\frac{\alpha_1}{2m_1}+ \dots + \frac{\alpha_s}{2m_1\cdots m_s}} \leq \mu C_s a^{\frac{A_s}{2}}.
\end{equation}
Hence, \eqref{E:toshow} results from \eqref{E:mixedmucs} and the following lemma.
\end{proof}
\begin{lemma}
If $0<\delta=a\delta'_{\mu}\leq \tilde\delta_{\mu}$ , then
\begin{equation}\label{E:aAs}
\mu C_s a^{\frac{A_s}{2}} <\delta^{\frac{1}{2}}.
\end{equation}
\end{lemma}
\begin{proof}
Since $0<\delta=a\delta'_{\mu}\leq \tilde
\delta_{\mu}\leq \Delta'_{\mu,s}$ by \eqref{E:tildeltamu},
it follows from \eqref{E:deltadoublemus} that
\begin{equation}\label{E:arange}
0 <a \leq \left(\frac{\delta'_{\mu}}{(2\mu C_s)^2} \right)^{\frac{1}{A_s-1}}.
\end{equation}
Since $A_s>1$, \eqref{E:arange} implies that
\begin{equation}\label{E:aineqdelta}
a^{A_s-1} \leq \frac{\delta'_{\mu}}{(2\mu C_s)^2}\quad \Longleftrightarrow \quad a^{A_s} \leq
\frac{a\delta'_{\mu}}{(2\mu C_s)^2}.
\end{equation}
Since $\delta=a\delta'_{\mu}$ , it follows \eqref{E:aineqdelta} that \eqref{E:aAs} holds. In fact, we have
\begin{equation*}
 \mu C_s a^{\frac{A_s}{2}} \leq
\frac{(a\delta'_{\mu})^{\frac{1}{2}}}{2} =\frac{\delta^{\frac{1}{2}}}{2} <\delta^{\frac{1}{2}}.
\end{equation*}
\end{proof}

\begin{proof}[Proof of \eqref{E:kslessJs}]
We now  show \eqref{E:kslessJs}. For simplicity we will write $k_i=k_i^s$, $1\leq i \leq s$, and
$\tau_s=\tau_s(p,\mu,\delta)$ in the remaining of the proof. Suppose that $k_s \geq J_s$. We obtain a
contradiction to \eqref{E:Deftaus} by showing that there exists $w_1$ such that
\begin{equation}\label{E:contraw1}
0 \ne |w_1| \lneqq  \tau_s \quad \mbox{and} \quad F_s^p(|w_1|) \geq C_{s,\delta}^p(\mu,|w_1|).
\end{equation}
Since $F_s^p(\tau_s)=|b_{s,J_s}\tau_s^{J_s}|$ and since we choose $J_s$ as the minimum one in Definition
\ref{D:DominantType}, it follows from \eqref{E:jkminimum} that there exists $\gamma_1$ with $0<\gamma_1 <\tau_s$
so that
\begin{equation*}
F_s^p(|w|)=\left|b_{s,J_s} w^{J_s}\right|, \quad \gamma_1 \leq |w| \leq \tau_s.
\end{equation*}
Since $C_{s,\delta}^p(\mu,\tau_s)=\mu|c_{s,K}(p) \tau_1^{k_1} \dots \tau_s^{k_s}|$ and since we choose $k_s$ as
the minimum one in Definition \ref{D:mixedtype}, by the same way we obtain $\gamma_2$ with $0<\gamma_2 <\tau_s$
so that
\begin{equation*}
c_{s,\delta}^p(\mu,|w|)=\mu \left|c_{s,K}(p) \tau_1^{k_1} \dots \tau_{s-1}^{k_{s-1}} w^{k_s} \right|, \quad
\gamma_2 \leq |w| \leq \tau_s.
\end{equation*}
Set $\gamma=\max\{\gamma_1,\gamma_2\}$. Note that $F_s^p$ and $C_{s,\delta}^p$ are monomials in $\gamma \leq |w|
\leq \tau_s$ with degree $J_s$ and $k_s$, respectively. If $J_s < k_s$, then since
$F_s^p(\tau_s)=C_{s,\delta}^p(\tau_s)$, it follows  that
\begin{equation*}
F_s^p(|w|)>C_{s,\delta}^p(\mu,|w|), \quad \mbox{for all} \quad \gamma \leq |w| <\tau_s.
\end{equation*}
If $J_s=k_s$, then since $F_s^p(|w|)$ and $C_{s,\delta}^p(|w|)$ are monomials in $\gamma\leq |w| \leq \tau_s$
with same degree and since $F_s^p(\tau_s)=C_{s,\delta}^p(\tau_s)$, it follows that
\begin{equation*}
F_s^p(|w|)=C_{s,\delta}^p(\mu,|w|), \quad \mbox{for all} \quad \gamma \leq |w| \leq \tau_s.
\end{equation*}
Hence, any $w_1$ with $ \gamma < |w_1| <\tau_s$ satisfies \eqref{E:contraw1}, which completes the proof.
\end{proof}

\section{Stability  of Approximate Systems}\label{S:uniform}
In this section we shall compare the sizes of the approximate systems at two distinct points with respect to
fixed $\mu>1$ and $\delta$ with $0<\delta\leq \tilde \delta_{\mu}$. We shall show that if $p$ and $p'$ is close
enough and have the same dominant and mixed types with respect to $\mu$ and $\delta$, then $\tau_s(p,\mu,\delta)$
and $\tau_s(p',\mu,\delta)$ are equal up to uniform constants. We will see that this stability property plays a
crucial role when we apply the covering argument to construct plurisubharmonic functions near the boundary in
section \ref{S:instrips}.

Let $p \in U$, $\mu>1$ and $0<\delta \leq \tilde \delta_{\mu}$, and let
$J_s$ and $K_s$ be the $s$-th dominant and mixed types at $p$ with respect to $\mu$ and $\delta$.
Let us denote
\begin{equation}\label{E:sigmaspdelta}
\sigma_s(p,\mu,\delta) =F_s^p\left( \tau_s(p,\mu,\delta) \right) =
C_{s,\delta}^p\left(\mu,\tau_s(p,\mu,\delta)\right), \quad 1 \leq s \leq n.
\end{equation}
It follows from \eqref{E:Deftaus} that 
\begin{align}
\sigma_s(p,\mu,\delta) &= |b_{s,J_s}(p)|\left[\tau_s(p,\mu,\delta)\right]^{J_s} \label{E:bsJssigma}\\
 &=\begin{cases}
 \mu |c_{s,K_s}(p)| \prod_{i=1}^s [\tau_i(p,\mu,\delta)]^{k_i^s}  &\quad \mbox{if $K_s\ne (0,\dots,0)$}
 \label{E:csKssigma}\\
 \delta^{\frac{1}{2}} &  \quad \mbox{if $K_s=(0,\dots,0)$}
 \end{cases},
\end{align}
and
\begin{equation}\label{E:sigmadelta}
[\sigma_s(p,\mu,\delta)]^2\geq \delta, \quad s=1,\dots,n.
\end{equation}
Furthermore, the coefficients of $f_s^p$ satisfy
\begin{equation}\label{E:bjest}
|b_{s,j}(p)| \leq \sigma_s (p,\mu,\delta) \left[\tau_s(p,\mu,\delta)\right]^{-j}, \quad j\geq 1,
\end{equation}
and
\begin{equation}\label{E:csalphaest}
|c_{s,\alpha}(p)| \leq \frac{1}{\mu} \sigma_s(p,\mu,\delta) \prod_{1\leq i \leq s}
\left[\tau_i(p,\mu,\delta)\right]^{-\alpha_i}, \quad \alpha \in \mathcal M_s
\end{equation}
Let $d_1,\dots,d_n$ be positive constants and let denote 
\begin{equation*}
R_{\mu,\delta}(p:d_1,\dots,d_n)= \{z \in \mathbb C^n : |z_i-p_i|\leq d_i \tau_i(p,\mu,\delta), \;\; 1\leq i \leq
n \}.
\end{equation*}
\begin{lemma}\label{L:CompareCoeff}
There exists a  constant $d$ with $0 < d <\frac{1}{2}$, depending only on $n$ and $m_1,\dots,m_n$, which
satisfies the following property: Let $p \in U$, $\mu>1$,  and $0<\delta \leq \tilde \delta_{\mu}$. Let $J_s$ and
$K_s=(k_1^s,\dots,k_n^s)$ denote the $s$-th dominant and mixed types at $p$ with respect to $\mu$ and $\delta$.
If $p' \in R_{\mu,\delta}(p:d,\dots,d)$, then
\begin{align}
|b_{s,J_s}(p')-b_{s,J_s}(p)| &\leq \frac{1}{4} |b_{s,J_s}(p)|, & 1\leq s \leq n  \label{E:comparebsJs}\\
|c_{s,K_s}(p')-c_{s,K_s}(p)| &\leq \frac{1}{4} |c_{s,K_s}(p)|, & 2 \leq s \leq n \label{E:comparecsKs}
\end{align}
\end{lemma}
\begin{proof}
Let $p'=p+u$ and let $D_s^j$ denote the partial derivatives $\frac{\partial^j}{\partial u_s^j}$. We at first show
\eqref{E:comparebsJs} when $s=1$. By Taylor's theorem we obtain that 
\begin{align}
\left|b_{1,J_1}(p')-b_{1,J_1}(p) \right| &= \frac{1}{J_1 !} \left|D_1^{J_1} f_1^p(u) - D_1^{J_1}f_1^p(0)
\right| \notag\\
& =\left| \sum_{j>J_1} b_{1,j}(p) {j \choose J_1} u_1^{j-J_1} \right| \label{E:taylorb1}
\end{align}
Here ${j \choose J_1}$ refers the binomial coefficient.
Set $d_1=(m_1+1)^{-1}2^{-(m_1+3)}$. If  $|u_1| \leq d_1\tau_1(p,\mu,\delta)$, then
\begin{align}
\left|b_{1,J_1}(p')-b_{1,J_1}(p) \right| & \leq \sum_{j>J_1} \sigma_1(p,\mu,\delta)
\left[\tau_1(p,\mu,\delta)\right]^{-j} {j \choose J_1}
[d_1\tau_1(p,\mu,\delta)]^{j-J_1} \notag \\
& = \sigma_1(p,\mu,\delta) [\tau_1(p,\mu,\delta)]^{-J_1}
\left(\sum_{j>J_1} {j \choose J_1} d_1^{j-J_1}\right) \notag \\
& = |b_{1,J_1}(p)| \left( \left(\frac{1}{1-d_1}\right)^{J_1+1}-1 \right). \label{E:b1J1prime}
\end{align}
In fact, we apply \eqref{E:bjest} to \eqref{E:taylorb1} to get the first inequality, and the third equality is
the result of \eqref{E:bsJssigma}. Note that since $J_1 \leq m_1$ by Proposition \ref{P:FinitePure} and since
$0<d_1 <\frac{1}{2}$, it follows that
\begin{equation}\label{E:powerprime}
\left(\frac{1}{1-d_1}\right)^{J_1+1}-1 =d_1\frac{\sum_{i=0}^{J_1}(1-d_1)^i}{(1-d_1)^{J_1+1}} \leq
(m_1+1)2^{m_1+1}d_1,
\end{equation}
Combine \eqref{E:b1J1prime} and \eqref{E:powerprime}  to obtain \eqref{E:comparebsJs}  for $s=1$.

To estimate $b_{s,J_s}(p')$, $s\geq 2$, consider 
\begin{align}
b_{s,J_s}(p')-b_{s,J_s}&(p)=   \sum_{j>J_s} b_{s,j}(p)
{j \choose J_s}u_s^{j-J_s} \notag\\
&-\sum_{\alpha \in \mathcal M_s \atop \alpha_s >J_s }c_{s,\alpha}(p)
u_1^{\alpha_1} \dots
 u_{s-1}^{\alpha_{s-1}} {\alpha_s \choose J_s} u_s^{\alpha_s-J_s}.
 \label{E:bsJsprime}
\end{align}
Let $A$ and $B$
denote the first and the second term in \eqref{E:bsJsprime}, respectively.
For each $s$ with $2\leq s \leq n$, set
\begin{equation}\label{E:esvalue}
d_s=\min \left \{ (m_s+1)^{-1}2^{-(m_s+s+3)} , (m_1\dots m_s+1)^{-1}2^{-(m_1\cdots m_s+3)} \right\}.
\end{equation}
If $|u_i| \leq d_i\tau_i(p,\mu,\delta)$, $1\leq i\leq s$, then by the same process used in \eqref{E:taylorb1},
\eqref{E:b1J1prime}, and \eqref{E:powerprime}, we have 
\begin{equation}\label{E:Aest}
|A| \leq |b_{s,J_s}(p)|\left\{ \left(\frac{1}{1-d_s}\right)^{J_s+1}-1\right\} \leq|b_{s,J_s}(p)| (m_s+1)2^{m_s+1}
d_s .
\end{equation}
To estimate $B$ we combine \eqref{E:csKssigma} and \eqref{E:csalphaest} with $|u_i|\leq d_i\tau_i(p,\mu,\delta)$,
$1\leq i \leq  s$, and then use a similar method used in the previous work to get 
\begin{align}
|B|&\leq |b_{s,J_s}(p)| \frac{1}{\mu}  \left( \prod_{i=1}^{s-1}\frac{1}{1-d_i} \right) \left\{
\left(\frac{1}{1-d_s} \right)^{J_s+1}-1 \right\} \notag \\
& \leq  |b_{s,J_s}(p)|(m_s+1)2^{m_s+s} d_s \label{E:Best}
\end{align}
In fact, the second line results from $\mu >1$, $0<d_i <\frac{1}{2}$, and  $1 \leq J_s \leq m_s$. Combining
\eqref{E:Aest} and \eqref{E:Best}, we obtain that
\begin{equation}\label{E:ABest}
|A|+|B| \leq |b_{s,J_s}(p)|(m_s+1) 2^{ m_s+s+1}  d_s.
\end{equation}
Hence,   it follows from \eqref{E:bsJsprime}, \eqref{E:esvalue}, and \eqref{E:ABest} that if $|u_i|\leq d_i
\tau_i(p,\mu,\delta)$, $1\leq i \leq s$, then
\begin{equation*}
|b_{s,J_s}(p') - b_{s,J_s}(p) |\leq  \frac{1}{4} |b_{s,J_s}(p)|.
\end{equation*}

 Now consider the case
 when $K_s \ne (0,\dots,0)$. It follows from Taylor's theorem that
\begin{equation}\label{E:diffcspprime}
|c_{s,K_s}(p')-c_{s,K_s}(p)|= \left| \sum_{\alpha \in \mathcal M_s \atop
\alpha_i>k_i^s} c_{s,\alpha} \prod_{i=1}^s
{\alpha_i \choose k_i^s} u_i^{\alpha_i-k_i^s} \right|,
\end{equation}
where $k_1+ \dots + k_s \geq 1$. We apply the similar process to in the previous ones to obtain that
\begin{align}
|c_{s,K_s}(p')-c_{s,K_s}(p)| &\leq \frac{\sigma_s(p,\mu,\delta)}{\mu} \prod_{1\leq i\leq s} \tau_i(p,\mu,\delta)^{-k_i^s}
\prod_{i=1}^s \left\{ \left( \frac{1}{1-d_i}\right)^{k_i^s+1}-1 \right\} \notag \\
& \leq |c_{s,K_s}(p)|\prod_{1\leq i \leq s} \left\{ \left( \frac{1}{1-d_i}\right)^{k_i^s+1}-1 \right\}. \notag\\
& \leq |c_{s,K_s}(p)|\prod_{1\leq i \leq s} \left\{ (m_1\cdots m_i+1)2^{m_1\cdots m_i +1}d_i \right\}
\label{E:csKsprimecs}
\end{align}
In fact, since $|u_i|\leq d_i \tau_i(p,\mu,\delta)$, $1\leq i \leq s$, we obtain the first inequality from
\eqref{E:csalphaest} and \eqref{E:diffcspprime}. The second inequality results from \eqref{E:csKssigma}. The
third inequality is obtained by $0 \leq k_i^s \leq m_1\dots m_i$,  $1 \leq i \leq s$ in Proposition
\ref{P:FiniteMixed}.

Therefore, it follows from \eqref{E:esvalue} and \eqref{E:csKsprimecs} that if  $|u_i|\leq d_i
\tau_i(p,\mu,\delta)$, $1\leq i \leq s$, then
\begin{equation*}
|c_{s,K_s}(p')| \leq \frac{1}{4} |c_{s,K_s}(p)|.
\end{equation*}
Let $d=\min \{d_s: 1 \leq s \leq n \}$, then this completes the proof.
\end{proof}
\begin{notation}
Let $a_1,\dots,a_n$ be positive constants, and for each $s$, $1\leq s\leq n$ we define 
\begin{equation}
R_{\mu,\delta}^s(p:a_1,\dots,a_s)= \{z \in \mathbb C^n : |z_i-p_i|\leq a_i \tau_i(p,\mu,\delta),  1\leq i \leq s
\}.
\end{equation}
When $s=n$, we will omit the superscript $n$ so that
\[
R_{\mu,\delta}(p:a_1,\dots,a_n)=R_{\mu,\delta}^n(p:a_1,\dots,a_n),
\]
and we write
\[
R_{\mu,\delta}(p)=R_{\mu,\delta}^n(p: 1,\dots,1).
\]
\end{notation}

\begin{definition}
Let $d$ be the constant in Lemma \ref{L:CompareCoeff}. Let us denote
\begin{equation}\label{E:tildeRdelta}
\tilde R_{\mu,\delta}(p)= R_{\mu,\delta}(p,d,\dots,d)
\end{equation}
and
\begin{equation}
\tilde \tau_s(p,\mu,\delta)=d \tau_s(p,\mu,\delta).
\end{equation}
\end{definition}

Let $A(t)>0$ and $B(t)>0$ be a function on a set $\mathbb T$, $t \in \mathbb T$. We shall use the notation $A(t)
\lesssim B(t)$, if there exists a positive constant $C$, independent of $t$, such that
\begin{equation}
A(t) \leq C B(t), \quad t \in \mathbb T.
\end{equation}
If $A(t) \lesssim B(t)$ and $B(t) \lesssim A(t)$, then we write $A(t) \approx B(t)$. In the following
proposition, the symbol $\approx$  means that $C$ is independent of $p \in U$, $\mu>1$, and $\delta$ with
$0<\delta\leq \tilde \delta_{\mu}$.

\begin{proposition}\label{P:compareinhat}
Let $ p ,p' \in U$, $\mu >1$, and $0<\delta\leq \tilde \delta_{\mu}$. Suppose that $p, p' \in U$ satisfy
\begin{equation}\label{E:interppprime}
\tilde R_{\mu,\delta}(p) \cap \tilde R_{\mu,\delta}(p') \ne \emptyset.
\end{equation}
If $p$ and $p'$ have the same $s$-th dominant and mixed types with respect to $\mu$ and $\delta$ for all $s$ with
$1\leq s \leq n$, then
\begin{equation}\label{E:taucompare}
\tau_s(p,\mu,\delta) \approx \tau_s(p,\mu,\delta), \quad 1 \leq s \leq n.
\end{equation}
\end{proposition}
\begin{proof}
Let $J_s=J_s(p,\mu,\delta)=J_s(p',\mu,\delta)$ and $K_s=K_s(p,\mu,\delta)=K_s(p',\mu,\delta)$. Choose any point
$p''$ so that $p'' \in \tilde R_{\delta}(p) \cap \tilde R_{\delta}(p')$. Lemma \ref{L:CompareCoeff} implies that
\begin{equation*}
|b_{s,J_s}(p'') - b_{s,J_s}(p)| \leq \frac{1}{4}|b_{s,J_s}(p)|, \quad \mbox{and} \quad
|b_{s,J_s}(p'')-b_{s,J_s}(p') | \leq \frac{1}{4}|b_{s,J_s}(p')|.
\end{equation*}
Hence, the triangular inequality gives us
\begin{align*}
||b_{s,J_s}(p)|- |b_{s,J_s}(p')||
&\leq |b_{s,J_s}(p)-b_{s,J_s}(p'')| + |b_{s,J_s}(p'')-b_{s,J_s}(p')| \notag \\
&\leq \frac{1}{4}|b_{s,J_s}(p)| + \frac{1}{4} |b_{s,J_s}(p')|, \notag
\end{align*}
which implies that
\begin{equation}\label{E:bsJsppprime}
\frac{3}{5} |b_{s,J_s}(p')| \leq |b_{s,J_s}(p)| \leq \frac{5}{3}|b_{s,J_s}(p')|, \quad 1\leq s \leq n.
\end{equation}
If $K_s \ne (0,\dots,0)$, then by the same way, we obtain that
\begin{equation}\label{E:approxcsKs}
\frac{3}{5} |c_{s,K_s}(p')| \leq |c_{s,K_s}(p)| \leq \frac{5}{3}|c_{s,K_s}(p')|, \quad 1 \leq s \leq n.
\end{equation}

We first show \eqref{E:taucompare} for $s=1$.  Since $J_1(p,\mu,\delta)=J_1(p',\mu,\delta)=J_1$, and since
$\sigma_1(p,\mu,\delta)=\sigma_1(p',\mu,\delta)=\delta^{\frac{1}{2}}$, it follows from \eqref{E:bsJssigma} that
\begin{equation}\label{E:tau1ppprime}
\tau_1(p,\mu,\delta)=\left(\frac{\delta^{\frac{1}{2}}}{|b_{1,J_1}(p)|} \right)^{\frac{1}{J_1}} \quad \mbox{and}
\quad \tau_1(p',\mu,\delta)=\left(\frac{\delta^{\frac{1}{2}}}{|b_{1,J_1}(p')|} \right)^{\frac{1}{J_1}}.
\end{equation}
Thus from \eqref{E:bsJsppprime} and \eqref{E:tau1ppprime}  we see that 
\begin{equation*}
\left(\frac{3}{5} \right)^{\frac{1}{J_1}} \leq \frac{\tau_1(p',\mu,\delta)}{\tau_1(p,\mu,\delta)} = \left(
\left|\frac{b_{1,J_1}(p)}{b_{1,J_1}(p')} \right| \right)^{\frac{1}{J_1}} \leq
\left(\frac{5}{3}\right)^{\frac{1}{J_1}}.
\end{equation*}
Since $J_1 \leq m_1$, we therefore obtain that $ \tau_1(p,\mu,\delta)\approx \tau_1(p',\mu,\delta)$.

Let $s\geq 2$. We assume inductively that if $p, p' \in U$ satisfy \eqref{E:interppprime} and if
$J_i(p,\mu,\delta)=J_i(p',\mu,\delta)$ and $K_i(p,\mu,\delta)=K_i(p,\mu,\delta)$ for all $i=1,\dots,s-1$, then
\eqref{E:taucompare} holds for $i=1,\dots,s-1$. If $K_s=(0,\dots,0)$, then since \eqref{E:bsJsppprime} holds for
$s$, we apply the same process  used in $s=1$ to obtain that $ \tau_s(p',\mu,\delta)\approx
\tau_s(p,\mu,\delta)$. If $K_s \ne (0,\dots,0)$, then it follows from \eqref{E:csKssigma} that
\begin{align*}
\left(\tau_s(p,\mu,\delta)\right)^{J_s-k_s^s}&=\frac{|c_{s,K_s}(p)|}
{|b_{s,J_s}(p)|}\left(\tau_1(p,\mu,\delta)\right)^{k_1^s} \dots
\left(\tau_{s-1}(p,\mu,\delta)\right)^{k_{s-1}^s}\\
\left(\tau_s(p',\mu,\delta)\right)^{J_s-k_s^s}&=\frac{|c_{s,K_s}(p')|}
{|b_{s,J_s}(p')|}\left(\tau_1(p',\mu,\delta)\right)^{k_1^s} \dots
\left(\tau_{s-1}(p',\mu,\delta)\right)^{k_{s-1}^s}.
\end{align*}
Note that  $0 \leq k_i^s \leq m_1\dots m_i$, $1 \leq i \leq s$ by \eqref{E:MixedFinite} and that $k_s^s <J_s\leq
m_s$ by \eqref{E:kslessJs}. Hence, since we assumed by induction that $\tau_i(p,\mu,\delta) \approx
\tau_i(p',\mu,\delta)$ for $1\leq i \leq s-1$, we combine \eqref{E:bsJsppprime} and \eqref{E:approxcsKs} to
obtain that $\tau_s(p,\mu,\delta) \approx \tau_s(p',\mu,\delta)$. This completes the proof.
\end{proof}
\section{Estimates of derivatives}\label{S:derivative}
In this section we prepare for the construction of local plurisubharmonic functions in section \ref{S:local}.
After shrinking $R_{\mu,\delta}(p)$, we will estimate the partial derivatives of $f_s$ on this small region in
terms of $\tau_s(p,\mu,\delta)$. In this section we shall fix a base point $p \in U$. When there is no confusion,
we simply write $\tau_s$ for each $\tau_s(p,\mu,\delta)$, $s=1,\dots,n$ with $\mu>1$, $0 < \delta \leq \tilde
\delta_{\mu}$. Let $u_i=z_i-p_i$ be the coordinates  centered at $p$. Note that $
\pa{f_s}{z_i}(z)=\pa{f_s^p}{u_i}(u),$ for $z=p+u$, $1\leq i \leq n$, $1\leq s \leq n.$
 Recall that 
\begin{equation*}
f^p_s(u) = \sum_{j\geq 1} b_{s,j} u_s^j + \sum_{\alpha\in \mathcal M_s} c_{s,\alpha} u_1^{\alpha_1} \dots
u_s^{\alpha_s}, \quad 1 \leq s \leq n,
\end{equation*}
where we omit $p$ in $b_{s,j}(p)$ and $c_{s,\alpha}(p)$ for simplicity. 
Recall that 
\begin{align*}
F_s^p(|u_s|)&=\sup_{j \geq 1} |b_{s,j}u_s^j| \\
C_{s,\delta}^p(\mu,|u_s|)&=\sup_{\alpha\in \mathcal M_s}\{\delta^{\frac{1}{2}}, \mu \cdot
|c_{s,\alpha}\tau_1^{\alpha_1} \dots \tau_{s-1}^{\alpha_{s-1}} u_s^{\alpha_s}| \}.
\end{align*}
Let us denote $J_s=J_s(p,\mu,\delta)$ and $K_s=K_s(p,\mu,\delta)$ for $s=1,\dots,n$. By omitting $(p,\mu,\delta)$
we will write \eqref{E:sigmaspdelta} as $ \sigma_s=F_s^p(\tau_s)=C_{s,\delta}^p(\mu,\tau_s)$, $s=1,\dots,n.$
In the following we shall fix a constant $a$ with $0<a <\frac{1}{8}$.
\begin{lemma}\label{L:DomInt}
Assume the same hypothesis in Lemma \ref{L:Mono} and let $x_0>0$. Then there exist an integer $k$, $1 \leq k \leq
N$, and points $x'$, $x''$ with 
\begin{equation*}
a^{(N+1)(2N+1)} x_0 \leq x' <x'' \leq x_0
\end{equation*}
such that $x' = a^{2N+1} x''$
and
$g(x)=g_{j_k}(x)$ for  $x'\leq x \leq x''$.
\end{lemma}
\begin{proof}
By Lemma \ref{L:Mono}, $g(x_0)=g_{j_0}(x_0)$ for some $j_0$. If $x_1=a^{2N+1}x_0$ and $g(x_1)=g_{j_0}(x_1)$, then
we are done. Just set $x'=x_1$ and $x''=x_0$. Otherwise, $g(x_1)=g_{j_1}(x_1)$ where $j_1 <j_0$. Repeat the
process until it terminates so that 
\begin{equation*}
g(x_k)=g_{j_k}(x_k), \quad g(x_{k+1}) =g_{j_k}(x_{k+1}).
\end{equation*}
Hence, $g(x)=g_{j_k}(x)$ for $x_{k+1} \leq x \leq x_k$. Furthermore, since $N \geq j_0
> j_1 > \dots \geq 1$, one  can obtain $k$ with $k \leq N+1$.
\end{proof}
\begin{lemma}\label{L:EstHprime}
Let $h(w)=\sum_{j=1}^{\infty}b_jw^j$ be holomorphic in  $|w|< \frac{3}{2}$. Suppose that $|b_m|\geq 1$ and that
there exists a constant $B$ such that $|b_j|\leq B$, $j=1,2,\dots$. Let $w_0$ be a point with $|w_0| \leq
B^{-1}$. Then there exist $w_1$, $w_2$, and an integer $k$ with $ 1\leq k \leq m$ so that
\begin{align}
&|w_1|=a|w_2|, \quad |w_1| \geq a^{(m+1)(2m+1)} |w_0|, \label{E:w1w2}\\
&F(|w|)|=|b_{k}w^{k}|, \quad |w_1| \leq |w| \leq |w_2|, \notag \\
&|h'(w)| \geq \frac{1}{2} |b_{k} k w^{k-1}|,  \quad|w_1| \leq |w| \leq |w_2|.\label{E:fprime}
\end{align}
\end{lemma}
\begin{proof}
By Lemma \ref{L:DomInt}, there exist $w'$  and  $w''$ with $|w'|=a^{2m+1}|w''|$,
and an integer $k$ with $1\leq k
\leq m$ such that
$|w'| \geq a^{(m+1)(2m+1)}|w_0|$
and
$F(|w|)=|b_kw^k|$ for $|w'| \leq |w| \leq |w''|$.
Hence, it follows that for any $j$ 
\begin{equation*}
|b_j(w')^j| \leq |b_k (w')^k| \quad \mbox{and} \quad |b_j(w'')^j| \leq |b_k(w'')^k|.
\end{equation*}
\noindent \textbf{Case 1}. If $|w| \leq  a^m |w''|$ and $j>k$, then 
\begin{align*}
\frac{|b_j j w^{j-1}|}{|b_k k w^{k-1}|} & = \frac{|b_j|}{|b_k|} \left(\frac{j}{k} \right)|w|^{j-k} \\
& \leq \frac{|b_j|}{|b_k|} \left(\frac{j}{k}\right) a^{m(j-k)} |w''|^{j-k}
= \frac{|b_j(w'')^j|}{|b_k (w'')^k|} \left(\frac{j}{k}\right) a^{m(j-k)}.
\end{align*}
Since $|b_j(w'')^j| \leq |b_k(w'')^k|$, it follows that 
\begin{equation*}
|b_j j w^{j-1}| \leq \left(\frac{j-k}{k} +1\right) a^{m(j-k)} |b_k k w^{k-1}|, \quad j>k.
\end{equation*}
Hence, if $|w| \leq a^m|w''|$, then 
\begin{align}
\sum_{j=k+1}^{\infty} |b_j jw^{j-1}| &\leq \left[\sum_{j=k+1}^{\infty} \left(\frac{j-k}{k}+1 \right)
a^{m(j-k)}\right]|b_k k w^{k-1}| \notag \\
&\leq \frac{a(2-a)}{(1-a)^2} |b_k k w^{k-1}| \label{E:jbig}
\end{align}
\noindent \textbf{Case 2}. If $|w| \geq a^{-m} |w'|$ and $j<k$, then 
in a similar way, we obtain that 
\begin{equation}\label{E:jsmall}
\sum_{j=1}^{k-1} |b_j j w^{j-1}|  \leq \frac{a}{1-a} |b_k k w^{k-1}|.
\end{equation}
Now let $w_1=a^{-m}w'$ and $w_2=a^{m} w''$. Since $|w'|=a^{2m+1}|w''|$, it follows that $|w_1|=a|w_2|$. Note that
$|w_1|$ and $|w_2|$ satisfy \eqref{E:w1w2}. Furthermore, since 
\begin{equation*}
|h'(w)|\geq |b_kkw^{k-1}| - \sum_{j\ne k} |b_j j w^{j-1}|,
\end{equation*}
by combining \eqref{E:jbig} and \eqref{E:jsmall}. we have 
\begin{equation*}
|h'(w)| \geq \left( 1-\frac{a(3-a)}{1-a)^2} \right) |b_k k w^{k-1}|, \quad |w_1|\leq |w| \leq |w_2|.
\end{equation*}
Therefore, since $0<a<\frac{1}{8}$, we obtain \eqref{E:fprime}.
\end{proof}
\begin{lemma}\label{L:basicestprime}
Let $p \in U$, $\mu>1$, and $0<\delta\leq \tilde \delta_{\mu}$, and let $d$ be the constant in  Lemma
\ref{L:CompareCoeff}. Let denote $N=\max\{m_k: 1\leq k \leq n\}$. For each $s$, $1\leq s \leq n$, there exist
constant $a_s>0$, and integers $k$ with $1 \leq k \leq J_s \leq m_s$ satisfying 
\begin{equation}
da^{(N+1)(2N+1)} \leq a_s \leq d a^N, \quad s=1,\dots,n,
\end{equation}
so that  if $ \frac{1}{2} a_s \tau_s \leq |u_s| \leq a_s \tau_s,$ then 
\begin{equation}\label{E:Fsmaxbs}
F_s^p(|u_s|) =|b_{s,k} u_s^{k}|,
\end{equation}
and 
\begin{equation}\label{E:partialfszs}
\left| \pa{f_s^p}{u_s}(u) \right| \geq \frac{1}{2} \left|b_{s,k}k u_s^{k-1}\right| - \sum_{\alpha \in \mathcal
M_s, \;\alpha_s \geq 1}  \left|c_{s,\alpha} u_1^{\alpha_1} \dots u_{s-1}^{\alpha_{s-1}} ( \alpha_s
u_s^{\alpha_s-1})\right|.
\end{equation}
\end{lemma}
\begin{proof}
It follows from \eqref{E:tausBs} that $\tilde \tau_s(p,\mu,\delta) \leq \tau_s(p,\mu,\delta) \leq B_s^{-1}$.
Replace $|w_0|$ by $\tilde \tau_s(p,\mu,\delta)$,  and $h(z)$ by  $ \sum_{j\geq 1} b_{s,j}(p)u_s^j$ in $f_s^p$,
and $\tilde \tau_s$, $1 \leq s \leq n$ in Lemma \ref{L:EstHprime}, respectively. to obtain \eqref{E:Fsmaxbs} and
\eqref{E:partialfszs}.
\end{proof}

\begin{proposition}\label{P:dominantpartial}
Let $d$ be the constant constructed in  Lemma \ref{L:CompareCoeff}.  There exist constants $\mu_1$ and $D$,
depending only on $n$ and $N$, with $\mu_1>1$ and $0<D<1$, so that the following properties hold: Suppose that
$\mu\geq \mu_1$ and let $a_s$, $1\leq s \leq n$ denote the constants constructed for $p$, $\mu$, and $\delta$, in
Lemma \ref{L:basicestprime}.
\begin{itemize}
\item [\textup{(i)}] If 
$ z \in R_{\mu,\delta}^{s-1}(p:a_1,\dots, a_{s-1})$ and $\frac{1}{2}a_s \tau_s \leq |z_s-p_s|\leq a_s\tau_s, $ 
then 
\begin{equation}\label{E:Estfszs}
\left|\pa{f_s}{z_s}(z) \right| \geq D \frac{\sigma_s}{a_s\tau_s}.
\end{equation}
\item [\textup{(ii)}]
 If $z \in
R_{\mu,\delta}^s(p: a_1,\dots,a_s)$, then 
\begin{equation}\label{E:mixedpartial}
\left| \pa{f_{s}}{z_i} (z)\right| \leq \frac{2^{n+1}}{\mu} \frac{\sigma_{s}}{a_i\tau_i}, \quad 1 \leq i <s.
\end{equation}
\end{itemize}
\end{proposition}

\begin{proof}
We choose the value of $\mu_1>1$ as 
\begin{equation}\label{E:Cond1Mu}
\mu_1 = 2^{n+N+3} d^{-N}a^{-N(N+1)(2N+1)},
\end{equation}
where $d$ is the constant in Lemma \ref{L:CompareCoeff} and $a$ is a fixed constant with $0<a<\frac{1}{8}$. Note
that $\mu_1$ depends only on $n$ and $N$.  Let $p\in U$, $\mu >\mu_1$, and $0<\delta\leq \tilde \delta_{\mu}$.
Let $A_1$ and $A_2$ denote the first and second term in \eqref{E:partialfszs}, respectively. It follows from
\eqref{E:Fsmaxbs} that for any $j\geq 1$ 
\begin{equation*}
|b_{s,k}u_s^k| \geq |b_{s,j}u_s^j|, \quad \frac{1}{2}a_s\tau_s \leq |u_s| \leq a_s \tau_s.
\end{equation*}
Hence, if $\frac{1}{2}a_s\tau_s\leq  |u_s| \leq a_s \tau_s $, then 
\begin{equation}\label{E:bskusk}
|b_{s,k}u_s^k| \geq \left|b_{s,k}\left(\frac{1}{2}a_s \tau_s \right)^k \right| \geq
\left|b_{s,J_s}\left(\frac{1}{2}a_s \tau_s \right)^{J_s} \right| = \left( \frac{1}{2} a_s \right)^{J_s}
\left|b_{s,J_s} \;\tau_s^{J_s} \right|.
\end{equation}
Since $J_s \leq m_s \leq N$, $a_s \geq d a^{(N+1)(2N+1)}$, and since $\sigma_s= F_s(\tau_s)= \left|b_{s,J_s} \;
\tau_s^{J_s} \right|$, 
it follows from  \eqref{E:bskusk} that if 
\begin{equation}\label{E:zscondition}
\frac{1}{2}a_s\tau_s\leq  |u_s| \leq a_s \tau_s
\end{equation}
then 
\begin{equation}\label{E:bszsEst}
|A_1|\geq \frac{1}{2|u_s|} |b_{s,k} u_s^k| \geq 2^{-(N+1)} d^Na^{N(N+1)(2N+1)} \frac{\sigma_s}{a_s\tau_s} .
\end{equation}

Now we estimate $A_2$. It follows from \eqref{E:csalphaest} that if $u$ satisfies $ |u_i| \leq a_i\tau_i, \quad
1\leq i \leq s, $ then 
\begin{align*}
|A_2|& \leq \frac{\sigma_s}{\mu (a_s\tau_s)} \sum_{\alpha \in \mathcal M_s, \;\alpha_s\geq 1} \left[\alpha_s
\left\{\prod_{i=1}^{s-1}\left(\frac{a_i\tau_i}{\tau_i}\right)^{\alpha_i}\right\} \left( \frac{a_s\tau_s}{\tau_s}
\right)^{\alpha_s}\right] \notag \\
& \leq \frac{\sigma_s}{\mu(a_s\tau_s)}\sum_{\alpha \in \mathcal M_s, \;\alpha_s \geq 1}
\left[\left( \prod_{i=1}^{s-1}a_i^{\alpha_i} \right) (\alpha_s a_s^{\alpha_s})\right] \notag\\
& \leq \frac{\sigma_s}{\mu(a_s\tau_s)}\left(\prod_{i=1}^{s-1}\frac{1}{1-a_i} \right)\frac{a_s}{(1-a_s)^2}.
\end{align*}
Since $\frac{1}{1-a_i} \leq 2$, $1 \leq i \leq s$, and since $a_s \leq d a^N \leq a^N$, it follows that if 
\begin{equation}\label{E:zicondition}
|u_i| \leq a_i\tau_i, \quad 1\leq i \leq s,
\end{equation}
then 
\begin{equation}\label{E:secondfszs}
|A_2| \leq \frac{2^{n+1}a^N}{\mu} \frac{\sigma_s}{a_s \tau_s}
\end{equation}
By combining the conditions, \eqref{E:zscondition} and \eqref{E:zicondition}, and the inequalities,
\eqref{E:bszsEst} and \eqref{E:secondfszs}, we have 
\begin{equation*}
\left|\pa{f_s^p}{u_s}(u) \right| \geq \left(2^{-(N+1)}d^Na^{N(N+1)(2N+1)} - \frac{2^{n+1}a^N}{\mu} \right)
\frac{\sigma_s}{a_s\tau_s}.
\end{equation*}
Set $D=2^{-(N+2)}d^Na^{N(N+1)(2N+1)}$, then it follows from \eqref{E:Cond1Mu} that \eqref{E:Estfszs} holds for
$\mu\geq \mu_1$ and $0<\delta \leq \tilde\delta_{\mu}$.

We now consider the mixed partial derivatives. We see that if $|u_j| \leq a_j\tau_j$, $1 \leq j \leq s$, then
\begin{align*}
\left| \pa{f_{s}^p}{u_i}(u)\right| &\leq \sum_{\alpha \in \mathcal M_s, \;\alpha_i \geq 1} \left| c_{s,\alpha}
u_1^{\alpha_1} \dots u_{i-1}^{\alpha_{i-1}}
u_{i+1}^{\alpha_{i+1}} \dots u_{s}^{\alpha_{s}}\right||(\alpha_i u_i^{\alpha_i-1})| \notag\\
& \leq \frac{1}{a_i \tau_i} \sum_{\alpha \in \mathcal M_s, \;\alpha_i\geq 1} \left|c_{s,\alpha} \alpha_i
\prod_{1\leq i \leq s}(a_i\tau_i)^{\alpha_i} \right|
\end{align*}
Since by \eqref{E:csalphaest}
\[
|c_{s,\alpha}| \leq \frac{\sigma_s}{\mu}\prod_{1\leq i \leq s} \tau_i^{-\alpha_i},
\]
it follows that 
\begin{equation}
\left| \pa{f_{s}^p}{u_i}(u)\right| \leq \frac{\sigma_{s}}{\mu(a_i\tau_i)} \left(\prod_{j\ne i \atop 1 \leq j
<s}\frac{1}{1-a_j}\right) \frac{a_i}{(1-a_i)^2}
\end{equation}
Therefore, we obtain \eqref{E:mixedpartial}.
\end{proof}
\section{Local Plurisubharmonic Functions}\label{S:local}
In this section we shall construct compactly supported plurisubharmonic functions with large Hessian near the
boundary, by adding well-chosen cut-off functions and then taking  compositions with convex functions. Let
$\chi(t)$ be a smooth function defined by
\begin{equation}\label{E:defchi}
\chi(t)=
\begin{cases}
\frac{1}{2n} t + \frac{3}{4} -\frac{1}{2n}, \quad & 0 \leq t \leq \frac{1}{2} \\
0 , \quad t \geq 1
\end{cases}
\end{equation}
and  $0 \leq \chi(t) \leq \frac{3}{4}$, $t\geq 0$. Let denote
\begin{equation}\label{E:constantM}
M= \sup\{ |\chi'(t)| + |\chi''(t)| : t\geq 0 \}.
\end{equation}
%
\begin{definition}
Let $\mu_1$ be the constant constructed in Proposition \ref{P:dominantpartial} and  let $p \in U$, $\mu\geq
\mu_1$, and $0<\delta\leq \tilde \delta_{\mu}$. Let $a_s$, $1 \leq s \leq n$, be the constants constructed for
$p$, $\mu$, $\delta$ in Lemma \ref{L:basicestprime}. For $p$, $\mu$, $\delta$, $s$, we define
\begin{equation}\label{E:Defchis}
\chi^p_{s,\mu,\delta}(z)=\chi\left(\frac{|z_s-p_s|^2}{(a_s\tau_s)^2}\right).
\end{equation}

\end{definition}
\begin{remark}
Note that $\Supp (\chi^p_{s,\mu,\delta}) \subset \{z \in \mathbb C^{n} : |z_s-p_s| \leq a_s \tau_s \}$,
and that
\begin{equation}\label{E:Hesschihalf}
\ddba \chi^p_{s,\mu,\delta}(z) (L,\bar L) = \frac{1}{2n} \frac{|t_s|^2}{(a_s\tau_s)^2}, \quad |z_s-p_s| \leq
\frac{1}{2} a_s \tau_s, \quad 1\leq s\leq n.
\end{equation}
Moreover, for all $L=t_1\pa{}{z_1}+ \dots + t_n \pa{}{z_n}$, $t_i \in \mathbb C$,
\begin{equation}\label{E:Hesschifull}
\left| \ddba \chi^p_{s,\mu,\delta}(z)(L,\bar L) \right| \leq M\frac{|t_s|^2}{(a_s\tau_s)^2}, \quad z \in \Supp
(\chi^p_{s,\mu,\delta}).
\end{equation}
\end{remark}
\begin{definition}
Let $\eta>0$ be a constant and define
\begin{equation}\label{E:sGpdelta}
G^p_{s,\mu,\eta,\delta}(z)= \frac{\eta\sum_{i=1}^s|f_i|^2}{\delta} + \sum_{i=1}^s \chi^p_{s,\mu,\delta}(z).
\end{equation}
When $s=n$, we write
\begin{equation}\label{E:Gpdelta}
G^p_{\mu,\eta,\delta}(z)=G^p_{n,\eta,\nu,\delta}(z).
\end{equation}
\end{definition}
\begin{theorem}\label{T:PluriCn}
Let $\mu_1$ be the constant in Proposition \ref{P:dominantpartial}. Then there exist constants, $\eta>1$ and
$\mu>\mu_1$, depending only on $n$ and $N$, so that  the following property holds: Let $a_1,\dots, a_s$ be the
constants constructed  for $p $, $\mu$, and $\delta$ with $0 <\delta \leq \tilde \delta_{\mu}$ in Lemma
\ref{L:basicestprime}. If 
$z \in R_{\mu,\delta}(p:a_1,\dots,a_n)$,
then 
\begin{equation}\label{E:desingular}
\ddba G^p_{\mu,\eta,\delta}(L,\bar L)(z) \geq \frac{1}{4n} \sum_{i=1}^n \frac{|t_i|^2}{(a_i\tau_i)^2}.
\end{equation}
\end{theorem}
\begin{lemma}\label{L:lamdelhessfs}
There exist constants, $\eta>1$ and $\mu\geq \mu_1$, only depending on $n$ and $N$, so that the following
properties hold: Let $s=1,\dots,n$, and let $a_1,\dots, a_s$ be the constants constructed  for $p \in U$, $\mu$,
and $0 <\delta \leq \tilde \delta_{\mu}$ in Lemma \ref{L:basicestprime}.  If $z$ satisfies
\begin{equation}\label{E:zcondlemma}
z \in R_{\delta}^{s-1}(p:a_1,\dots,a_{s-1}), \quad  \mbox{and} \quad \frac{1}{2} a_s \tau_s \leq |z_s-p_s| \leq
a_s \tau_s,
\end{equation}
then for all $L=t_1\pa{}{z_1}+ \dots + t_n \pa{}{z_n}$, $t_i \in \mathbb C$
\begin{equation}\label{E:lamdelhessfs}
\ddba \left(\frac{\eta |f_s|^2}{\delta}\right)(z) (L,\bar{L}) \geq (M+1)\frac{|t_s|^2}{(a_s\tau_s)^2}-
\frac{1}{4n^2} \sum_{i=1}^{s-1} \frac{|t_i|^2}{(a_i\tau_i)^2}
\end{equation}
\end{lemma}
\begin{proof}
We shall prove \eqref{E:lamdelhessfs} only for the case when $s\geq 2$. By the same way we can prove it for
$s=1$. Let $p \in U$, $\mu\geq \mu_1$, and $0<\delta\leq \tilde \delta_{\mu}$, and let denote
$\sigma_s=\sigma_s(p,\mu,\delta)$ for simplicity. Since $\delta^{-1}\sigma_s^2 \geq 1$ by \eqref{E:sigmadelta},
it follows that
\begin{align}
\ddba\left(\frac{\eta |f_s|^2}{\delta}\right)(z)(L,\bar{L})&= \frac{\eta \sigma_s^2}{\delta}\left|\sum_{i=1}^s
\sigma_s^{-1}\pa{f_s}{z_i} t_i\right| \notag \\
&\geq \eta \left| \sum_{i=1}^{s} \sigma_s^{-1} \pa{f_s}{z_i}t_i \right|^2.\label{E:fssquare}
\end{align}
To estimate the right side in\eqref{E:fssquare} we consider
\begin{align}
\left|\sum_{i=1}^s \sigma_s^{-1}\pa{f_s}{z_i} t_i\right|^2 &= \left| \left(\sum_{i=1}^{s-1} \sigma_s^{-1}
\pa{f_s}{z_i}t_i \right)+ \sigma_s^{-1}\pa{f_s}{z_s} t_s\right|^2 \notag\\
& \geq 2 \Real \sum_{i=1}^{s-1} \left( \sigma_s^{-1} \pa{f_s}{z_i} t_i \right) \left(\overline{\sigma_s^{-1}
\pa{f_s}{z_s}t_s}\right) + \left|\sigma_s^{-1} \pa{f_s}{z_s}t_s \right|^2 .\label{E:fssqdiv}
\end{align}
Cauchy inequality gives us that  for $1\leq i \leq s$
\begin{equation}\label{E:realsigmasfs}
\left| 2 \Real \left(\sigma_s^{-1}\pa{f_s}{z_s} t_i \right)
\left( \overline{ \sigma_s^{-1} \pa{f_s}{z_s} t_s}
\right) \right| \leq 2n \left|\sigma_s^{-1} \pa{f_s}{z_i}t_i \right|^2 - \frac{1}{2n} \left|
\sigma_s^{-1}\pa{f_s}{z_s} t_s \right|^2.
\end{equation}
Hence, we combine  \eqref{E:fssquare}, \eqref{E:fssqdiv}, and \eqref{E:realsigmasfs} to obtain that 
\begin{align}
\ddba\left(\frac{\eta |f_s|^2}{\delta}\right)(z) (L,\bar L) &\geq
\eta \left(1-\frac{s}{2n} \right) \left|\sigma_s^{-1}\pa{f_s}{z_s} \right|^2|t_s|^2\notag \\
& \quad - \eta (2n) \sum_{i=1}^{s-1} \left|\sigma_s^{-1}\pa{f_s}{z_i} \right| |t_i|^2.\label{E:flamdel}
\end{align}
Note that $\frac{s}{2n} \leq \frac{1}{2}$ for $1\leq s \leq n$.  Proposition \ref{P:dominantpartial} and
\eqref{E:flamdel} imply that  if $\mu \geq \mu_1$ and $z$ satisfies \eqref{E:zcondlemma}, then 
\begin{align}
\ddba \left(\frac{\eta  |f_s|^2}{\delta}\right)(z) (L,\bar{L}) &\geq \frac{\eta}{2} D^2
\frac{|t_s|^2}{(a_s\tau_s)^2} \notag\\
& \:- \frac{\eta}{\mu^2} (2n) 2^{2(n+2)} a^N \sum_{i=1}^{s-1} \frac{|t_i|^2}{(a_i\tau_i)^2}. \label{E:etaddba}
\end{align}
Let us choose the values of $\eta$ and $\mu$ so that $\eta = 2(M+1)D^{-2}$ and $\mu  = \max \left\{\mu_1,\;
\left(n^3 2^{2n+5} \eta\right)^{\frac{1}{2}} \right\}$. Note that those numbers, $\eta$ and $\mu$, depend on $n$
and $N$. Since $\frac{\eta}{2} D^2 =M + 1$  and $\frac{\eta}{\mu^2} (2n) 2^{2(n+2)} \leq \frac{1}{4n^2}$,
 it follows from \eqref{E:etaddba} that we obtain \eqref{E:lamdelhessfs}.
\end{proof}

\begin{proof}[Proof of Theorem \ref{T:PluriCn}]
Let $\eta$ and $\mu$ be chosen in Lemma \ref{L:lamdelhessfs}. We at first prove that the following estimate holds
for any $s=1,\dots,n$: If 
$z \in R_{\mu,\delta}^s(p:a_1,\dots,a_s)$,
then 
\begin{equation}\label{E:induction}
\ddba G^p_{s,\mu,\eta,\delta}(L,\bar L)(z) \geq \sum_{i=1}^{s} \left(\frac{1}{2n} - \frac{s-i}{4n^2} \right)
\frac{|t_i|^2}{(a_i\tau_i)^2}.
\end{equation}
Since $\ddba \left(\frac{\eta|f_1|^2}{\delta} \right)(L,\bar L)(z) \geq 0$, it follows from \eqref{E:sGpdelta}
and \eqref{E:Hesschihalf} that 
\begin{equation*}
\ddba G^p_{1,\mu,\eta,\delta}(L,\bar L)(z) \geq \frac{1}{2n} \frac{|t_1|^2}{(a_1\tau_1)^2}, \quad |z_1-p_1| \leq
\frac{1}{2} a_1\tau_1.
\end{equation*}
If $\frac{1}{2}a_1\tau_1 \leq |z_1-p_1| \leq a_1\tau_1$, then we have
\begin{align*}
\ddba G^p_{1,\mu,\eta,\delta} (L,\bar L) (z) & = \ddba\left(\frac{\eta
|f_1|^2}{\delta}\right)(L,\bar L)  + \ddba \chi^p_{1,\mu,\delta}(L,\bar L)(z) \notag\\
&\geq (1+M) \frac{|t_1|^2}{(a_1\tau_1)^2} - M \frac{|t_1|^2}{(a_1\tau_1)^2} \notag \\
& \geq \frac{1}{2n} \frac{|t_1|^2}{(a_1\tau_1)^2}.
\end{align*}
In fact, the second inequality results from Lemma \ref{L:lamdelhessfs} and \eqref{E:Hesschifull}. 
Hence,  we showed  \eqref{E:induction} for $s=1$.

Let $1 \leq s \leq n-1$ and we assume inductively that \eqref{E:induction} holds for any $i$ with $1\leq i \leq
s$. 
It follows from \eqref{E:sGpdelta} that 
\begin{equation*}
G^p_{s+1,\mu,\eta,\delta}(z)=G_{s,\mu,\eta,\delta}^p(z) + \frac{\eta|f_{s+1}(z)|^2}{\delta} +
\chi^p_{s+1,\mu,\delta}(z).
\end{equation*}
Applying \eqref{E:induction} for $i$ with $1\leq i \leq s$, we obtain that if $z \in
R_{\mu,\delta}^{s}(p:a_1,\dots,a_s)$, then
\begin{align*}
\ddba G^p_{s+1,\mu,\eta,\delta}(L,\bar L)(z) &\geq \sum_{i=1}^s \left(\frac{1}{2n}-
\frac{s-i}{4n^2}\right)\frac{|t_i|^2}{(a_i\tau_i)^2}\notag\\
& \; +\ddba \left( \frac{\eta |f_{s+1}|^2}{\delta} \right)(L,\bar L) (z) + \ddba \chi^p_{s+1,\mu,\delta}(L, \bar
L)(z).
\end{align*}
Hence, \eqref{E:induction} holds for $s+1$ if $z \in R_{\mu,\delta}^{s}(p:a_1,\dots,a_s)$ and $0 \leq
|z_{s+1}-p_{s+1}| \leq \frac{1}{2} a_{s+1}\tau_{s+1}$. Here are details:  
\begin{align*}
\ddba G^p_{s+1,\mu,\eta,\delta}(L,\bar L) &\geq \sum_{i=1}^s \left(\frac{1}{2n}-
\frac{s-i}{4n^2}\right)\frac{|t_i|^2}{(a_i\tau_i)^2} +\frac{1}{2n}
\frac{|t_{n+1}|^2}{(a_{s+1}\tau_{s+1})^2} \notag \\
& \geq \sum_{i=1}^{s+1}\left(\frac{1}{2n}-\frac{s+1-i}{4n^2}\right) \frac{|t_i|^2}{(a_i\tau_i)^2}.
\end{align*}
In fact, the first inequality  follows from \eqref{E:Hesschihalf}. 
Furthermore, \eqref{E:induction} also holds for $s+1$ if $z \in R_{\mu,\delta}^{s}(p:a_1,\dots,a_s)$ and
$\frac{1}{2}a_{s+1}\tau_{s+1} \leq |z_{s+1}-p_{s+1}| \leq a_{s+1}\tau_{s+1}$. Here are details:  
\begin{align}
\ddba G^p_{s+1,\mu,\eta,\delta}(L,\bar L)(z) & \geq \sum_{i=1}^s \left(\frac{1}{2n}-
\frac{s-i}{4n^2}\right)\frac{|t_i|^2}{(a_i\tau_i)^2} \notag \\
& \quad +(M+1) \frac{|t_{s+1}|^2}{(a_{s+1}\tau_{s+1})^2} - \frac{1}{4n^2} \sum_{i=1}^s
\frac{|t_i|^2}{(a_i\tau_i)^2} \notag \\
& \quad - M \frac{|t_{s+1}|^2}{(a_{s+1}\tau_{s+1})^2}. \notag
\end{align}
In fact, the terms in the second line come from Lemma \ref{L:lamdelhessfs}, and the term in the third line
results from \eqref{E:Hesschifull}.

 Therefore, \eqref{E:induction} holds  for all $s=1,\dots,n$, by induction. In
particular, if $s=n$, then since $\frac{n-i}{4n^2} \leq \frac{1}{4n}$ for $1 \leq i \leq n$, we obtain
\eqref{E:desingular}.
\end{proof}

\begin{notation}
In the following we shall fix $\mu$ and $\eta$ chosen in Theorem \ref{T:PluriCn}. Let $p \in U$ and $0<\delta\leq
\tilde \delta_{\mu}$, and let $a_1,\dots,a_n$ be the constants constructed  for $p \in U$, $\mu$, and $\delta$ in
Lemma \ref{L:basicestprime}. For simplicity we shall omit $\mu$ and $\eta$ which appear in
$\tau_s(p,\mu,\delta)$, $\chi_{s,\mu,\delta}^p(z)$, $G^p_{\mu,\eta,\delta}(z)$, and $R_{\mu,\delta}(p:
a_1,\dots,a_n)$, to write
\[
\tau_s(p,\delta), \quad \chi^p_{s,\delta}(z), \quad G^p_{\delta}(z), \quad  R_{\delta}(p:a_1,\dots,a_n).
\]
\end{notation}
We now modify $G^p_{\delta}$ to construct a plurisubharmonic function with compact support in $\bar \Omega$ near
the boundary of $\Omega$. Let $z'=(z_1,\dots,z_n,z_{n+1})$ denote the coordinates of $\mathbb C^{n+1}$ and let
denote $L'_i= \pa{}{z_i} - \pa{r}{z_i} \left(\pa{r}{z_{n+1}} \right)^{-1}\pa{}{z_{n+1}}$ for $1\leq i \leq n$,
and let $L'_{n+1}  = \pa{}{z_{n+1}}$. Set  $L' =\sum_{i=1}^{n+1} t_i L'_i$, where $t_i \in \mathbb C$ for $1 \leq
i \leq n+1$.
Since $\pa{r}{z_{n+1}}=\frac{1}{2}$ and since $L'_i(r) \equiv 0$, it follows that $L'(r)=\frac{1}{2}t_{n+1}$ and
$|L'|^2 \approx \sum_{i=1}^{n+1}|t_i|^2$. 
Set
\begin{equation}
S(\delta)=\{z' \in \bar \Omega : -\delta \leq r(z') \leq 0 \}
\end{equation}
\begin{theorem}\label{T:localpluri}
There exist  small constants $c>d>0$, and a constant $C>0$, depending only on $n$ and $N$, so that the following
property holds: If $p \in U$ and $0 <\delta<\delta_{\mu}$, there exists a smooth plurisubharmonic function
$g_{p,\delta}$ in $\bar \Omega$ that satisfies 
\begin{itemize}
\item [\textup{(i)}] $0\leq g_{p,\delta}(z') \leq 1$, $z' \in \bar \Omega$, and $g_{p,\delta}$ is supported in
\begin{equation}\label{E:Suppgdelta}
S(c\delta) \cap \{z' \in \mathbb C^{n+1} : z \in R_{\delta}(p:a_1,\dots,a_n) \}
\end{equation}
\item [\textup{(ii)}] if $z'$ satisfies 
\begin{equation}\label{E:zprimesmallbox}
z' \in S(d\delta) \cap \{ z' \in \mathbb C^{n+1} : z \in R_{\delta}\left(p:\frac{a_1}{2}, \dots,
\frac{a_n}{2}\right) \}
\end{equation}
then 
\begin{equation}\label{E:gdeltaHess}
\ddba g_{p,\delta} (L',\bar L')(z') \geq C\left( \sum_{i=1}^n \frac{|t_i|^2}{\tau_i^2} +
\frac{|t_{n+1}|^2}{\delta^2} \right)
\end{equation}
\end{itemize}
\end{theorem}

\begin{lemma}\label{L:GdeltaHess}
Let 
\begin{equation}
\tilde G^p_{\delta}(z')=\exp\left(\frac{4\eta\; r(z')}{\delta}\right)+\sum_{s=1}^n \chi^p_{s,\delta}(z).
\end{equation}
If $z'$ satisfies 
\begin{equation}\label{E:Condizprime}
-\frac{\log 4}{4\eta}\delta \leq r(z') \leq 0, \quad \mbox{and} \quad z \in R_{\delta}(p:a_1,\dots, a_n),
\end{equation}
then 
\begin{equation}\label{E:ddbatildaG}
\ddba \tilde G^p_{\delta}(L',\bar L')(z') \geq \frac{1}{4n} \sum_{i=1}^n\frac{|t_i|^2}{(a_i\tau_i)^2} +(\eta
\delta^{-1})^2 |t_{n+1}|^2.
\end{equation}
\end{lemma}
\begin{proof}
Note that 
\begin{align*}
\ddba \tilde G^p_{\delta}(L',\bar L')(z') &=e^{4\eta \;\delta^{-1}\; r} \left[
\left(\frac{4\eta}{\delta}\right)^2 \dee r \wedge \dba r
+ \frac{4\eta}{\delta} \ddba r\right] (L',\bar L')(z') \notag \\
& \; + \ddba \left( \sum_{s=1}^n \chi^p_{s,\delta} \right)(L,\bar L)(z).
\end{align*}
Since $(\dee r \wedge \dba r) (L',\bar L')=\frac{1}{4}|t_{n+1}|^2$ and $\ddba r = \ddba ( \sum_{s=1}^{n} |f_s|^2
)$, and since  $\frac{1}{4} \leq e^{4\eta\; \delta^{-1}\;r(z')} \leq 1$  for $z'$ satisfying
\eqref{E:Condizprime}, it follows that 
\begin{align*}
\ddba \tilde G^p_{\delta}(L',\bar L')(z') &\geq (\eta \delta^{-1})^2|t_{n+1}|^2 \notag \\
&\;+ \ddba \left(\eta \delta^{-1}\sum_{s=1}^n|f_s|^2 +\sum_{i=1}^n \chi_{i,\delta}\right)(L,\bar L)(z)\notag \\
&= (\eta \delta^{-1})^2|t_{n+1}|^2 + \ddba G_{\delta}(L,\bar L) (z).
\end{align*}
Since  $z \in R_{\delta}(p:a_1,\dots,a_n)$,  we obtain \eqref{E:ddbatildaG} by Theorem \ref{T:PluriCn}.
\end{proof}

\begin{proof}[Proof of Theorem \ref{T:localpluri}]
Let $P$ be a convex increasing function such that $P(t)=0$ for $t \leq \frac{1}{4}$, $P(t)>0$ for $t>
\frac{1}{4}$, and $P(1)=1$. Let $g_{p,\delta}$ be the function defined  by 
\begin{equation*}
g_{p,\delta}(z')=P\left(\tilde G^p_{\delta}(z') -\frac{3}{4}n \right).
\end{equation*}
Let us choose $c>0$ and $d>0$ as $c=\frac{\log 4}{4\eta}$ and $d=\frac{\log 6-\log 5}{4\eta}$, where $\eta$ is
fixed as in Theorem \ref{T:PluriCn}. Since $\eta$ depends only on $n$ and $N$, it follows that $c$ and $d$ depend
only on $n$ and $N$. If $z' \in \bar \Omega$ and $|z_i| > a_i \tau_i$ for some $i$, then $g_{p,\delta}(z')=0$. In
fact, if $z' \in \bar \Omega$ and $|z_i| > a_i \tau_i$ for some $i$, then
\begin{align*}
\tilde G^p_{\delta}(z') - \frac{3}{4} n&\leq 1+(n-1)\max\{\chi(t):t\geq 0\} -\frac{3}{4}n \notag\\
& =1+(n-1)\frac{3}{4}-\frac{3}{4}n=\frac{1}{4}.
\end{align*}
Similarly, if $r(z') < -c\delta$, then $g_{p,\delta}(z')=0$. In fact, since $e^{4\eta\delta^{-1}}r(z') \leq
\frac{1}{4}$, it follows that $\tilde G^p_{\delta}(z') - \frac{3}{4}n \leq \frac{1}{4} + n \max\{ \chi(t): t\geq
0\} -\frac{3}{4}n = \frac{1}{4}$. Combining the previous two results, we obtain that the support of
$g_{p,\delta}$ is contained in the set described in \eqref{E:Suppgdelta}. Furthermore, if $r(z')\leq 0$, then $
G_{\delta}^p(z')-\frac{3}{4}n \leq 1 + n \max\{\chi(t) : t\geq 0\} -\frac{3}{4} =1$. Since $P(t)$ is increasing
and $P(1)=1$, it follows that $0\leq g_{p,\delta}(z') \leq 1$ for $z' \in \bar \Omega$. Therefore, we showed the
first part of the theorem.

Now we want to show $g_{p,\delta}$ is a plurisubharmonic function in $\bar \Omega$. Note that $P$ is convex increasing and
that if $z' \in \bar \Omega$ satisfies $g_{p.\delta}\ne 0$ then $z'$ is contained the set in \eqref{E:Suppgdelta}. Hence,
it follows from Lemma \ref{L:GdeltaHess} that $g_{p,\delta}$ is plurisubharmonic in $\bar \Omega$.

To prove the second part of the theorem, we consider the following: if  $z' \in \bar \Omega$ satisfies
\eqref{E:zprimesmallbox}, then since $e^{4\eta\delta^{-1}}r(z') \geq \frac{5}{6}$, we have 
$\tilde G^p_{\delta}(z') -\frac{1}{4}n \geq \frac{1}{3} >\frac{1}{4}$. 
Since $P'(t)\geq C'$, $t\geq \frac{1}{3}$ for
some constant $C'>0$, Lemma \ref{L:GdeltaHess} implies that if $z'$ satisfies \eqref{E:zprimesmallbox} then
\begin{equation}
\ddba g_{p,\delta} (L',L')(z') \geq \frac{C'}{4n} \sum_{i=1}^n \frac{|t_i|^2}{(a_i\tau_i)^2} +C'(\eta \delta^{-1})^2
|t_{n+1}|^2.
\end{equation}
Note that $\eta$ depends only on $n$ and $N$ and that $a_1,\dots,a_n$ are bounded below by a constant depending
only on $n$ and $N$. Therefore, \eqref{E:gdeltaHess} holds.
\end{proof}
\section{Plurisubharmonic Functions in Strips}\label{S:instrips}

In this section we will construct bounded plurisubharmonic functions with large Hessian near the boundary. Recall
the constant $d$ with $0 < d <\frac{1}{2}$  in Lemma \ref{L:CompareCoeff}, whose size depends only on $n$ and
$N$. We also recall  the constants, $a_1,\dots,a_n$, in Lemma \ref{L:basicestprime}, corresponding to each $p\in
U$ and $\delta$ with $0<\delta\leq \delta_{\mu}$,  such that $da^{(N+1)(2N+1)}\leq a_s \leq da^N$, $1\leq s\leq
n$, where $a$ is a fixed constant with $0<a<\frac{1}{8}$.
\begin{definition}
For each $p\in U$ and $\delta$ with $0<\delta \leq \delta_0$ define
\begin{equation}
\hat \tau_s(p,\delta)=a_s\tau_s(p,\delta), \quad 1 \leq s \leq n,
\end{equation}
and
\begin{equation}
\hat R_{\delta}(p) = R_{\delta}(p:a_1,\dots,a_n).
\end{equation}
\end{definition}
\begin{remark}
Since $a_s \leq da^N <\frac{1}{8}d$, it follows that
\begin{equation}\label{E:hatincludetilde}
\hat R_{\delta}(p) \subset \tilde R_{\delta}(p)=R_{\delta}(p:d,\dots,d).
\end{equation}
Furthermore, since  $a_s$ are bounded above and below by uniform constants, depending only on $n$ and $N$, it
follows that
\begin{equation}\label{E:Comparehat}
\tau_s(p,\delta) \approx \tilde \tau_s(p,\delta) \approx \hat \tau_s(p,\delta), \quad 1 \leq s \leq n.
\end{equation}
\end{remark}
Since the value of  $\mu$ was fixed in Theorem \ref{T:PluriCn}, we simply write $ J_s(p,\delta)$ and
$K_s(p,\delta)$ for $J_s(p,\mu,\delta)$ and $K_s(p,\mu,\delta)$, respectively, for $p \in U$,  $0<\delta\leq
\tilde\delta_{\mu}$, and $1 \leq s \leq n$.  Let $\mathcal T(p,\delta)$ be the collection of the dominant and
mixed types at $p$ with respect to $\delta$, that is,
\begin{equation}
\mathcal T(p,\delta)= \{ J_s(p,\delta), K_s(p,\delta) : p \in U, 0 <\delta \leq \tilde \delta_{\mu}, 1 \leq s
\leq n\}.
\end{equation}
It follows from Proposition \ref{P:FinitePure} and \ref{P:FiniteMixed} that  $\mathfrak T= \{\mathcal T(p,\delta)
: p \in U, 0 <\delta \leq \tilde \delta_{\mu} \}$ is finite. Fix any $\mathcal T \in \mathfrak T$ and $\delta$
with $0 <\delta \leq \tilde \delta_{\mu}$, and define $U_{\mathcal T, \delta}$ by
\begin{equation*}
U_{\mathcal T,\delta}=\{p \in U : \mathcal T(p,\delta)=\mathcal T \}.
\end{equation*}
We now follow a similar argument used  in \cite{Ca87} and \cite{Ca89}. It follows from \eqref{E:lutaus} that
$\tau_s(p,\delta) \geq B_s^{-1}\delta^{\frac{1}{2}}$, $1 \leq s \leq n$, where $B_s$ is independent of $p$ and
$\delta$. Since we chose $U$ as a bounded set in Section \ref{S:regularcoordinate}, it follows from
\eqref{E:Comparehat} that there exists a selection of $p_k \in U_{\mathcal T,\delta}$, $k=1,\dots,N'$, where $N'$
depends on $\delta$ and $\mathcal T$, so that
\begin{equation*}
U_{\mathcal T,\delta} \subset \bigcup_{k=1}^{N'} \frac{1}{2} \hat R_{\delta}(p_k),
\end{equation*}
and
\begin{equation}\label{E:pknotin14}
p_k \notin \frac{1}{4} \hat R_{\delta}(p_j), \quad j=1, \dots, k-1, \quad k=1,\dots,N'.
\end{equation}
\begin{lemma}\label{L:finitecover}
Let $p$ be any point in $\{p_k : k=1,\dots,N'\}$ and set
\begin{equation*}
E=\{j:\hat R_{\delta}(p) \cap \hat R_{\delta}(p_j)\ne \emptyset, j=1,\dots, N' \}.
\end{equation*}
Then there exists an integer $M_{\mathcal T}$, depending only on $\mathcal T$, independent of $p$ and $\delta$,
so that
\begin{equation*}
\# (E) \leq M_{\mathcal T}.
\end{equation*}
\end{lemma}
\begin{proof}
Suppose that $j \in E$. Since $\hat R_{\delta}(p) \subset \tilde R_{\delta}(p)$ and $\hat R_{\delta}(p_j) \subset
\tilde R_{\delta}(p_j)$ by \eqref{E:hatincludetilde} and since we assumed that $\mathcal T(p,\delta)=\mathcal
T(p_j,\delta)=\mathcal T$, it follows from Proposition \ref{P:compareinhat} that if $\hat R_{\delta}(p) \cap \hat
R_{\delta}(p_j) \ne \emptyset$, then
\begin{equation*}
\tau_s(p,\delta) \approx \tau_s(p_j,\delta), \quad 1 \leq s \leq n.
\end{equation*}
Since $\tau_s(p,\delta) \approx \hat \tau_s(p,\delta)$ and $\tau_s(p_j,\delta) \approx \hat \tau_s(p_j,\delta)$
by \eqref{E:Comparehat}, it follows that there exists a large constant $D$, independent of $p$ and $\delta$ such
that for $j \in E$,
\begin{equation*}
\hat R_{\delta}(p_j) \subset D\hat R_{\delta}(p).
\end{equation*}

Since $\hat \tau_s(p,\delta) \approx \hat \tau_s(p_j,\delta)$, $j\in E$, $1\leq s \leq n$, there exists a small
constant $d>0$, independent of $p$ and $\delta$, such that each polydisc $\frac{1}{4}\hat R_{\delta}(p_j)$, $j
\in E$, contains a polydisc $P_{\delta}(p_j)$, centered at $p_j$, defined by
\begin{equation}
P_{\delta}(p_j)=\{z \in \mathbb C^n: |z_s-(p_j)_s|\leq d\hat \tau_s(p,\delta)\},
\end{equation}
where $p_j=((p_j)_1,\dots,(p_j)_n)$. We now choose any pair of $j, k \in E$ with $k>j$.  Since $p_k \notin
\frac{1}{4} \hat R_{\delta}(p_j)$  and  the $s$-th sides of $P_{\delta}(p_j)$ and $P_{\delta}(p_k)$ are equal to
$d\hat \tau_s(p,\delta)$, we can shrink $d$, independent of $p$ and $\delta$, so that $P_{\delta}(p_j) \cap
P_{\delta}(p_k) = \emptyset$. Since the volume of $P_{\delta}(p_j)$ equals the volume of $d\hat R_{\delta}(p)$
and since the volumes of $D\hat R_{\delta}(p)$ and $d\hat R_{\delta}(p)$ are equal up to a constant, independent
of $p$ and $\delta$, it guarantees that there exists an integer $M_{\mathcal T}$, independent of $p$ and $\delta$
such that $\# (E) \leq M_{\mathcal T}$.
\end{proof}
Define $\lambda_{\mathcal T, \delta}(z')=\sum_{j=1}^N g_{p_j,\delta}(z')$ for each $\mathcal T \in \mathfrak T$
and $\delta$ with $0<\delta \leq \tilde \delta_{\mu}$. Let us denote
\begin{equation}
t_{\mathcal T,\delta}=\sup\{\tau_s(p,\delta) : p \in U_{\mathcal T,\delta}, \;s=1,\dots,n \},
\end{equation}
and
\begin{equation}
t_{\delta} = \sup \{\tau_s(p,\delta) : p \in U, \; s=1,\dots,n\}.
\end{equation}
It follows from Theorem \ref{T:localpluri} and Lemma \ref{L:finitecover} that  $\lambda_{\mathcal T,\delta}(z')$
a well-defined plurisubharmonic functions in $\bar \Omega$ for each $\delta$ and $\mathcal T$,  such that
\begin{itemize}
\item [\textup{(i)}] $0 \leq \lambda_{\mathcal T,\delta}(z')  \leq M_{\mathcal T}$, $z' \in \bar \Omega$, and
$\Supp( \lambda_{\mathcal T,\delta}) \subset S(c\delta)$, 
\item [\textup{(ii)}] if $z' \in S(d\delta) \cap \left\{ z' \in \mathbb C^{n+1} : z \in U_{\mathcal T} \right\}$
then 
\begin{equation*}
\ddba \lambda_{\mathcal T,\delta} (L',\bar L')(z') \geq M_{\mathcal T} C\frac{1}{\left(t_{\mathcal
T,\delta}\right)^2}\;|L'|^2,
\end{equation*}
\end{itemize}
where the constants, $c$,  $d$ and $C$, are constructed in Theorem \ref{T:localpluri}. We define $
\lambda_{\delta}(z')= \sum_{\mathcal T\in \mathfrak T} \lambda_{\mathcal T,\delta}(z')$ for each  $\delta$ with
$0<\delta\leq \tilde \delta_{\mu}$.  Since $M_{\mathcal T}$ and $C$ are independent of $z'$ and $\delta$ and
since $\mathfrak T$ is finite, it follows that $\lambda_{\delta}$ is a well-defined smooth plurisubharmonic
function in $\bar \Omega$ such that
\begin{itemize}
\item [(i)] $0 \leq \lambda_{\delta}(z') \lesssim 1$, $z' \in \bar \Omega$,

\item [(ii] if $z' \in S(d\delta) \cap \{(z,z_{n+1}) : z \in U \}$, then
$\ddba \lambda_{\delta} (L',\bar L')(z') \gtrsim \frac{1}{(t_{\delta})^2}\;|L'|^2$.
\end{itemize}
Since by Remark \ref{R:UniSmall}, $\tau_s(p,\delta) \lesssim \delta^{\frac{1}{2m_1\dots m_s}}$, $1\leq s \leq n$,
uniformly in $p\in U$ and $\delta$ with $0<\delta\leq \tilde \delta_{\mu}$,  we obtain the following theorem.
\begin{theorem}\label{T:main}
Let $\Omega$ be a regular coordinate domain at the origin in $\mathbb C^{n+1}$  with \eqref{E:amultiplicity} .
Then a subelliptic estimate holds at the origin of order $\epsilon = \frac{1}{2m_1\dots m_n}$.
\end{theorem}

\begin{remark}
Since the construction of $\tau_s(p,\delta)$ is closely related to the disc type, introduced in \cite{Ca83}, we
expect that the $\epsilon$ in Theorem \ref{T:main} is the sharpest subelliptic gain on a regular coordinate
domain.
\end{remark}

Since we use the inequality in \eqref{E:desingular} to construct plurisubharmonic functions, and since other
processes do not cause any difference, we obtain monotonicity for subelliptic estimates. A similar argument
appears in \cite{Cho08m}.

\begin{corollary}
Let $\Omega$ be a regular coordinate domain at the origin in $\mathbb C^{n+1}$ with \eqref{E:amultiplicity}. Let
$\rho(z,\bar z)$ be a plurisubharmonic function near the origin in $\mathbb C^n$, and let $\Omega'$ denote a
pseudoconvex domain  defined near the origin in $\mathbb C^{n+1}$ by
\begin{equation*}
\Omega'= \{z' \in \mathbb C^{n+1} : \Real z_{n+1}+\rho(z,\bar z)<0\}.
\end{equation*}
If there is a neighborhood $U$ of the origin in $\mathbb C^n$ so that the complex Hessian of $\rho(z,\bar z)$ is
bigger than the one of $\sum_{s=1}^n|f_s(z)|^2$ on $U$, then a subelliptic estimate holds for $\Omega'$ of order
$\epsilon$, where $\epsilon$ is obtained in Theorem \ref{T:main}.
\end{corollary}


\end{document}